\documentclass[a4paper,12pt]{amsart}

\usepackage{amsfonts}
\usepackage{amsmath}
\usepackage{amssymb}
\usepackage{graphicx}

\usepackage[usenames]{color}
\usepackage[colorlinks]{hyperref}

\newtheorem{thm}{Theorem}[section]

\newtheorem{prop}[thm]{Proposition}
\newtheorem{cor}[thm]{Corollary}

\theoremstyle{definition}

\theoremstyle{remark}
\newtheorem{rem}{Remark}[section]
\newtheorem{defn}{Definition}

\numberwithin{equation}{section}

\def\d{\mathrm d}

\begin{document}

\title[Wave equations for the fractional S-L operator: singularities]{Wave equations for the fractional Sturm-Liouville operator with singular coefficients}

\author[M. Ruzhansky]{Michael Ruzhansky}
\address{
  Michael Ruzhansky:
  \endgraf
  Department of Mathematics: Analysis, Logic and Discrete Mathematics
  \endgraf
  Ghent University, Krijgslaan 281, Building S8, B 9000 Ghent
  \endgraf
  Belgium
  \endgraf
  and
  \endgraf
  School of Mathematical Sciences
  \endgraf
  Queen Mary University of London
  \endgraf
  United Kingdom
  \endgraf
  {\it E-mail address} {\rm michael.ruzhansky@ugent.be}
}

\author[M. Sebih]{Mohammed Elamine Sebih}
\address{
  Mohammed Elamine Sebih:
  \endgraf
  Laboratory of Geomatics, Ecology and Environment (LGEO2E)
  \endgraf
  University Mustapha Stambouli of Mascara, 29000 Mascara
  \endgraf
  Algeria
  \endgraf
  {\it E-mail address} {\rm sebihmed@gmail.com, ma.sebih@univ-mascara.dz}
}

\author[A. Yeskermessuly ]{Alibek Yeskermessuly}
\address{
  Alibek Yeskermessuly:
  \endgraf   
  Faculty of Natural Sciences and Informatization
  \endgraf
  Altynsarin Arkalyk Pedagogical Institute, Auelbekov, 17, 110300 Arkalyk, Kazakhstan
  \endgraf  
  {\it E-mail address:} {\rm alibek.yeskermessuly@gmail.com}
  }

\thanks{This research was funded by the FWO Odysseus 1 grant G.0H94.18N: Analysis and Partial Differential Equations, and by the Methusalem programme of the Ghent University Special Research Fund (BOF) (Grant number 01M01021). MR is also supported by EPSRC grants EP/R003025/2 and EP/V005529/1.}


\begin{abstract}
In this paper, we consider the wave equation for the fractional Sturm-Liouville operator with lower order terms and singular coefficients and data. We prove that the problem has a very weak solution. Furthermore, we prove the uniqueness in an appropriate sense and the consistency of the very weak solution concept with the classical theory.
\end{abstract}

\maketitle


\section{Introduction}
In the present paper, our investigation is devoted to the wave equation generated by the fractional Sturm-Liouville operator involving lower order terms and singularities in the coefficients and the data. That is, for $s\geq 0$ and $T>0$, we study the equation
\begin{equation}\label{Equation}
    \partial_{t}^{2}u(t,x) + \mathcal{L}^su(t,x) + a(x)u(t,x) + b(x)u_{t}(t,x) =0, \quad (t,x)\in [0,T]\times (0,1),
\end{equation}
subject to the initial conditions
\begin{equation}\label{Initial conditions}
    u(0,x)=u_{0}(x), \quad u_{t}(0,x)=u_{1}(x), \quad x\in (0,1),
\end{equation}
and boundary conditions
\begin{equation}\label{boundary conditions}
    u(t,0)=u(t,1)=0, \quad t\in [0,T],
\end{equation}
where $a$, $b$ are assumed to be non-negative and $\mathcal{L}^s$ is the fractional differential operator associated to the Sturm-Liouville operator defined by
\begin{equation}\label{S-L operator}
    \mathcal{L}u(t,x):=-\partial_{x}^{2}u(t,x) + q(x)u(t,x),
\end{equation}
for a real valued function $q$.

The Sturm-Liouville operator with singular potential was studied by Savchuk and Shkalikov in \cite{SS99}. In this work, asymptotic estimates for the eigenvalues and corresponding eigenfunctions, when the operator includes singular potential were obtained. We also cite \cite{NS99}, \cite{Sav01}, \cite{SS06} and \cite{SV15} where the Sturm-Liouville operator with distributional potentials was explored.

So, our aim in the present work is to study the well-posedness of the initial/boundary problem \eqref{Equation}-\eqref{boundary conditions}, where the spatially dependent coefficients $a,b$ and $q$ and the initial data $u_0$ and $u_1$ are allowed to be non-regular functions having in mind the Dirac delta function and its powers. We do this study under the framework of the concept of very weak solutions. Our reasons to get into this framework lies in the fact that when the equation under consideration contains products of distributional terms, it is no longer possible to pose the problem in the distributional framework. This is related to the well known work of Schwartz \cite{Sch54} about the impossibility of multiplication of distributions.

In order to give a neat solution to this problem, the concept of very weak solutions was introduced in \cite{GR15} for the analysis of second order hyperbolic equations with singular coefficients. Later on, this concept of solutions has been developed for a number of problems. We cite for instance \cite{RT17a}, \cite{RT17b}, \cite{MRT19}, \cite{ART19}, \cite{ARST21a}, \cite{ARST21b}, \cite{ARST21c}, \cite{CRT21}, \cite{CRT22a}, \cite{CRT22b}, \cite{SW22} and \cite{CDRT23} to mention only few. In \cite{ARST21a}, \cite{ARST21b}, \cite{ARST21c}, \cite{CRT21}, \cite{CRT22a} and \cite{CRT22b}, arguments were based on energy methods. In the recent works \cite{RSY22}, \cite{RY22} and \cite{RY23}, the existence of solutions to initial/boundary value problems for the Sturm-Liouville operator including various types of time-dependent singular coefficients was considered. In these works, separation of variables techniques \cite{GSS17} were possible in order to obtain explicit formulas to the classical solutions. Our aim in the present paper is to combine separation of variables techniques with energy methods in order to extend the results obtained in \cite{RSY22} and \cite{RY22}, firstly by considering the fractional Sturm-Liouville operator instead of the classical one, and secondly, by including more terms in the equation under consideration. Most importantly, we allow coefficients to depend on space, so that the previous separation of variables methods do not readily apply.

The paper is organised as follows. After some preliminaries about the classical and the fractional Sturm-Liouville operator and about Duhamel's principle, we establish in Section \ref{Classical case}, energy estimates in the regular case, which are key in proving existence and uniqueness of very weak solutions. We treat two cases. The general case when $s\geq 0$ and the case $s=1$. In Section \ref{VW well-posed}, we introduce the notion of very weak solutions adapted to our considered problem \eqref{Equation}-\eqref{boundary conditions} and we prove that it is very weakly well-posed. Section \ref{Consistency} is devoted to showing the consistency of the concept of very weak solutions with the classical theory.

\section{Preliminaries}\label{Prelim}
The following notations and notions will be frequently used throughout this paper.
\subsection{Notation}
\begin{itemize}
    \item By the notation $f\lesssim g$, we mean that there exists a positive constant $C$, such that $f \leq Cg$.
    \item We also define
    \begin{equation*}
        \Vert u(t,\cdot)\Vert_{s} := \Vert u(t,\cdot)\Vert_{L^2} + \Vert \mathcal{L}^{\frac{s}{2}}u(t,\cdot)\Vert_{L^2} + \Vert u_t(t,\cdot)\Vert_{L^2}.
    \end{equation*}
\end{itemize}
\subsection{Classical Sturm-Liouville operator}
Here we present some spectral properties of the Sturm-Liouville operator obtained in \cite{SS99}. We consider the Sturm-Liouville operator $\mathcal{L}$ generated in the interval (0,1) by the differential expression
\begin{equation}\label{St-L}
    \mathcal{L}y:=-\frac{d^2}{dx^2}y+q(x)y,
\end{equation}
with boundary conditions
\begin{equation}\label{Dirihle}
    y(0)=y(1)=0. 
\end{equation}
We first consider the real potential $q$ satisfying
\begin{equation}\label{con-q}
    q(x)=\nu'(x),\text{~such that~} \nu\in L^2(0,1).
\end{equation}
The eigenvalue problem $\mathcal{L}y=\lambda y$ with Dirichlet boundary conditions \eqref{Dirihle} has eigenvalues 
\begin{equation}\label{e-val}
    \lambda_n=(\pi n)^2(1+o(n^{-1})),\qquad n=1,2,...,
\end{equation}
and corresponding eigenfunctions
\begin{equation}\label{sol-SL}
    \Tilde{\phi}_n(x)=r_n(x)\sin\theta_n(x)=r_n(x)\sin(\sqrt{\lambda_n}x +\eta_n(x)), 
\end{equation}
where
$$r_n(x)=\exp{\left(-\int\limits_0^x \nu(s)\cos2\theta_n(s)ds+o(1)\right)}=1+o(1)\quad \text{as }n\to \infty,$$
$$\eta_n(x)=o(1), \quad n=1,2,...$$

Since $\lambda_n$ are real and according to \eqref{con-q}, $\nu$ is a real valued function. Then, the eigenfunctions $\Tilde{\phi}_n$ are real.\\
The first derivatives of $\Tilde{\phi_n}$ have the following form
\begin{equation}\label{phi-der}
    \Tilde{\phi}'_n(x)=\sqrt{\lambda_n}r_n(x)\cos(\theta_n(x))+\nu(x)\Tilde{\phi}_n(x).
\end{equation}
By Theorem 2 in \cite{Sav01} we have that
\begin{equation}\label{phi-sav}
\Tilde{\phi}_n(x)=\sin{\sqrt{\lambda_n}x}+\psi_n(x), \quad n=1,2,...,\quad \sum\limits_{n=1}^\infty \|\psi_n\|^2\leq C \int\limits_0^1|\nu(x)|^2dx.
\end{equation}
The estimate for $\|\Tilde{\phi}_n\|_{L^2}$ follows by taking the $L^2$ norm in \eqref{sol-SL} and by proceeding as follows
\begin{eqnarray}\label{est-high}
\|\Tilde{\phi}_n\|^2_{L^2}&=&\int\limits_0^1\left|r_n(x)\sin\left(\lambda_n^{\frac{1}{2}}x+\eta_n(x)\right)\right|^2dx\leq \int\limits_0^1\left|r_n(x)\right|^2dx\nonumber\\
&\leq& \int\limits_0^1\left|\exp{\left(-\int\limits_0^x \nu(s)\cos{2\theta_n(s)}ds-\frac{1}{2}\frac{1}{\sqrt{\lambda_n}}\int\limits_0^x\nu^2(s)\sin{2\theta_n(s)}ds\right)}\right|^2dx\nonumber\\
&\lesssim& \int\limits_0^1\exp{\left(2\int\limits_0^x|\nu(s)|ds+\frac{1}{\sqrt{\lambda_n}}\int\limits_0^x|\nu^2(s)|ds\right)}dx\nonumber\\
&\lesssim& \exp{2\left(\|\nu\|_{L^2}+\lambda_n^{-\frac{1}{2}}\|\nu\|^2_{L^2}\right)}<\infty.
\end{eqnarray}
In addition, according to Theorem 4 in \cite{SS99}, we get
\begin{equation}\label{est_low}
  \Tilde{\phi}_n(x)=\sin(\pi nx)+o(1),  
\end{equation}
for sufficiently large $n$. Combining this with \eqref{sol-SL}, we see that there exist a constant $C_0>0$, such that
\begin{equation}\label{low-est}
0<C_0\leq \|\Tilde{\phi}_n\|_{L^2}<\infty \quad \text{for all } n.
\end{equation}
The family of eigenfunctions of the operator $\mathcal{L}$ form an orthogonal basis in $L^2(0,1)$. Moreover, we will normalize them and denote
\begin{equation}\label{norm-phi}
  \phi_n(x)=\frac{\Tilde{\phi}_n(x)}{\sqrt{\langle \Tilde{\phi}_n,\Tilde{\phi}_n}\rangle}=\frac{\Tilde{\phi}_n(x)}{\|\Tilde{\phi}_n\|_{L^2}}.  
\end{equation}

\subsection{Fractional Sturm-Liouville operator}
\begin{defn}\label{S.T fractional}
    Let $\{\lambda_k ,\phi_k\}_{k=1,\cdots, \infty}$ be the family of eigenvalues and corresponding eigenfunctions to the classical Sturm-Liouville operator as defined above. Then, for $s\in\mathbb R$, $\mathcal{L}^s$ is defined in the sense that:
    \begin{equation}
        \mathcal{L}^s \phi_k := \lambda^s \phi_k,
    \end{equation}
    for all $k=1,\cdots,$.
\end{defn}
In other words, $\mathcal{L}^s$ is defined to be the operator having the family $\{\lambda_k^s ,\phi_k\}_{k=1,\cdots,}$ as family of eigenvalues and corresponding eigenfunctions.

\begin{prop}\label{prop1}
Let $\mathcal{L}$ be the Sturm-Liouville operator generated in the interval (0,1) by the differential expression \eqref{St-L}
with boundary conditions \eqref{Dirihle}. Assume that $(f, g) \in L^2(0,1)\times L^2(0,1)$ with $(\mathcal{L}^s f,\mathcal{L}^s g) \in L^2(0,1)\times L^2(0,1)$. Then
    \begin{equation}\label{L^s}
        \langle \mathcal{L}^sf,g\rangle_{L^2}=\langle f,\mathcal{L}^sg\rangle_{L^2} \text{~for any~} s\in \mathbb R,
    \end{equation}
    and
    \begin{equation}\label{L^s+s'}
        \mathcal{L}^{s+s'}f=\mathcal{L}^s \left(\mathcal{L}^{s'}f\right) \text{~for~} s,s'\in \mathbb R.
    \end{equation}
 
\end{prop}
\begin{proof}
Since for ${n=1,2,\cdots,}$ the eigenfunctions $\phi_n$ of the Sturm-Liouville operator are orthonormal in $L^2(0,1)$ and using the fact that the operator $\mathcal{L}$ is self-adjoint (\cite{KS79}) and using eigenfunction expansions for $f,\, g \in L^2(0,1)$, we obtain
\begin{eqnarray}\label{L^s1}
\langle \mathcal{L}^sf,g\rangle_{L^2}&=&\int\limits_0^1 \sum\limits_{n=1}^\infty\lambda_n^s f_n\phi_n(x) \sum\limits_{m=1}^\infty g_m\phi_m(x)dx=\sum\limits_{n=1}^\infty\sum\limits_{m=1}^\infty\int\limits_0^1 \lambda_n^s f_n\phi_n(x) g_m\phi_m(x)dx\nonumber \\
 &=&\sum\limits_{n=1}^\infty  \lambda_n^sf_n g_n\int\limits_0^1\phi^2_n(x)dx=\sum\limits_{n=1}^\infty \lambda_n^s f_n g_n,    
\end{eqnarray}
where
$$f_n=\int\limits_0^1f(x)\phi_n(x)dx, \quad g_n=\int\limits_0^1g(x)\phi_n(x)dx.$$
On the other hand, we similarly get
\begin{equation}\label{L^s2}
    \langle f, \mathcal{L}^s g\rangle_{L^2} = \sum\limits_{n=1}^\infty f_n \lambda_n^s g_n=\sum\limits_{n=1}^\infty \lambda_n^s f_n  g_n.
\end{equation}
This proves the first statement. For \eqref{L^s+s'}, the second statement of the proposition, we have
\begin{eqnarray*}
\mathcal{L}^{s+s'}f&=&\sum\limits_{n=1}^\infty \mathcal{L}^{s+s'}f_n\phi_n(x)=\sum\limits_{n=1}^\infty\lambda_n^{s}\left(\lambda_n^{s'}f_n\phi_n(x)\right)\\
&=&\sum\limits_{n=1}^\infty\mathcal{L}^{s}\left(\mathcal{L}^{s'}f_n\phi_n(x)\right)=\mathcal{L}^s\left(\mathcal{L}^{s'}f\right),  
\end{eqnarray*}
completing the proof.
\end{proof}

\subsection{Sobolev spaces and embeddings}
We define the Sobolev spaces $W^s_\mathcal{L}$ associated to $\mathcal{L}^s$, for any $s \in \mathbb{R}$, as the space
\begin{equation}\label{Sobolev space def}
    W^s_\mathcal{L}(0,1):=\left\{f\in \mathcal{D}'_\mathcal{L}(0,1):\,\mathcal{L}^{s/2}f\in L^2(0,1)\right\},
\end{equation}
with the norm $\|f\|_{W^s_\mathcal{L}}:=\|\mathcal{L}^{s/2}f\|_{L^2}$. The global space of distributions $\mathcal{D}'_\mathcal{L}(0,1)$ is defined as follows.

The space $C^\infty_\mathcal{L}(0,1):=\mathrm{Dom}(\mathcal{L}^\infty)$ is called the space of test functions for $\mathcal{L}$, where we define 
$$\mathrm{Dom}(\mathcal{L}^\infty):=\bigcap\limits_{m=1}^\infty \mathrm{Dom}(\mathcal{L}^m),$$
where $\mathrm{Dom}(\mathcal{L}^m)$ is the domain of the operator $\mathcal{L}^m$, in turn defined as
$$\mathrm{Dom}(\mathcal{L}^m):=\left\{f\in L^2(0,1): \mathcal{L}^j f\in \mathrm{Dom}(\mathcal{L}),\,\, j=0,1,2,...,m-1\right\}.$$
The Fréchet topology of $C^\infty_\mathcal{L}(0,1)$ is given by the family of norms 
\begin{equation}\label{frechet}
    \|\phi\|_{C^m_\mathcal{L}}:=\max\limits_{j\leq m}\|\mathcal{L}^j\phi\|_{L^2(0,1)},\quad m\in \mathbb{N}_0,\,\, \phi\in C^\infty_\mathcal{L}(0,1).
\end{equation}
The space of $\mathcal{L}$-distributions
$$\mathcal{D}'_\mathcal{L}(0,1):=\mathbf{L}\left(C^\infty_\mathcal{L}(0,1),\mathbb{C}\right)$$
is the space of all linear continuous functionals on $C^\infty_\mathcal{L}(0,1)$. For $\omega \in \mathcal{D}'_\mathcal{L}(0,1)$ and $\phi\in C^\infty_\mathcal{L}(0,1)$, we shall write 
$$\omega(\phi)=\langle \omega, \phi\rangle.$$
For any $\psi \in C^\infty_\mathcal{L}(0,1)$, the functional 
$$C^\infty_\mathcal{L}(0,1)\ni \phi \mapsto \int\limits_0^1 \psi(x)\phi(x)dx$$
is an $\mathcal{L}$-distribution, which gives an embedding $\psi \in C^\infty_\mathcal{L}(0,1)\hookrightarrow \mathcal{D}'_\mathcal{L}(0,1)$.


\begin{prop}[Sobolev embeddings]\label{Prop w^-s} 
Let $0<s\in \mathbb{R}$ and $f\in W^s_{\mathcal{L}}(0,1)$ satisfying the boundary conditions $f(0)=f(1)=0$. Then, we have the continuous inclusions 
\begin{equation}\label{Sobolev embeddings}
    W_{\mathcal{L}}^{s}(0,1) \subset L^2 (0,1)\subset W_{\mathcal{L}}^{-s}(0,1).
\end{equation}
That is, for any $f\in W_{\mathcal{L}}^{s}(0,1)$, we have $f\in L^2 (0,1)$ and accordingly $f\in W_{\mathcal{L}}^{-s}(0,1)$. Moreover, there exist positive constants $C_1$, $C_2$ independent of $f$ such that
\begin{equation}\label{Sobolev estimates 1}
    \Vert f\Vert_{W_{\mathcal{L}}^{-s}} \leq C_1 \Vert f\Vert_{L^2},
\end{equation}
and
\begin{equation}\label{Sobolev estimates 2}
    \Vert f\Vert_{L^2} \leq C_2 \Vert f\Vert_{W_{\mathcal{L}}^{s}}.
\end{equation}
\end{prop}

\begin{proof}
The first embedding is a direct consequence of the definition of $W_{\mathcal{L}}^{s}(0,1)$ (see \eqref{Sobolev space def}). Let us prove the second statement. According to \eqref{e-val}, the eigenvalues of the operator $\mathcal{L}$ are outside the unit ball, then 
$$\lambda_n^{-\frac{s}{2}}\leq 1$$
for all $n=1,2,\cdots,$. This leads to the following estimate
\begin{eqnarray*}
    \|f\|_{W^{-s}_{\mathcal{L}}}^2&=&\|\mathcal{L}^{-\frac{s}{2}}f\|_{L^2}^2=\int\limits_0^1\left|\sum\limits_{n=1}^\infty\mathcal{L}^{-\frac{s}{2}}f_{n}\phi_n(x)\right|^2dx=\int\limits_0^1\left|\sum\limits_{n=1}^\infty\lambda_n^{-\frac{s}{2}}f_{n}\phi_n(x)\right|^2dx\\
&\lesssim&\sum\limits_{n=1}^\infty\int\limits_0^1\left|\lambda_n^{-\frac{s}{2}}f_{n}\phi_n(x)\right|^2dx=\sum\limits_{n=1}^\infty\left|\lambda_n^{-\frac{s}{2}}f_{n}\right|^2 \leq \sum\limits_{n=1}^\infty |f_{n}|^2=\|f\|_{L^2}^2,
\end{eqnarray*}
completing the proof.
\end{proof}

\subsection{Duhamel's principle}
Throughout this paper, we will often use the following version of Duhamel's principle for which the proof is given. For more details and applications about Duhamel's principle, we refer the reader to \cite{ER18}. Let us consider the following initial/boundary problem,
\begin{equation}\label{Equation Duhamel}
    \left\lbrace
    \begin{array}{l}
    u_{tt}(t,x)+Lu(t,x)+\alpha(x)u_{t}(t,x)=f(t,x) ,~~~(t,x)\in\left(0,\infty\right)\times (0,1),\\
    u(0,x)=u_{0}(x),\,\,\, u_{t}(0,x)=u_{1}(x),\,\,\, x\in(0,1),\\
    u(t,0)=u(t,1)=0,\,\,\,t\in (0,\infty),
    \end{array}
    \right.
\end{equation}
for a given function $\alpha$ and $L$ is a linear partial differential operator acting over the spatial variable $x$.

\begin{prop}\label{Prop Duhamel}
The solution to the initial/boundary problem (\ref{Equation Duhamel}) is given by
\begin{equation}\label{Sol Duhamel}
    u(t,x)= w(t,x) + \int_0^t v(t,x;\tau)\d \tau,
\end{equation}
where $w(t,x)$ is the solution to the homogeneous problem
\begin{equation}\label{Homog eqn Duhamel}
    \left\lbrace
    \begin{array}{l}
    w_{tt}(t,x)+Lw(t,x)+\alpha(x)w_{t}(t,x)=0 ,~~~(t,x)\in\left(0,\infty\right)\times (0,1),\\
    w(0,x)=u_{0}(x),\,\,\, w_{t}(0,x)=u_{1}(x),\,\,\, x\in(0,1),\\
    w(t,0)=w(t,1)=0,\,\,\,t\in (0,\infty),
    \end{array}
    \right.
\end{equation}
and for fixed $\tau\in\left(0,\infty\right)$, $v(t,x;\tau)$ solves the auxiliary problem
\begin{equation}\label{Aux eqn Duhamel}
    \left\lbrace
    \begin{array}{l}
    v_{tt}(t,x;\tau)+Lv(t,x;\tau)+a(x)v_{t}(t,x;\tau)=0 ,~~~(t,x)\in\left(\tau,\infty\right)\times (0,1),\\
    v(\tau,x;\tau)=0,\,\,\, v_{t}(\tau,x;\tau)=f(\tau,x),\,\,\, x\in(0,1),\\
    v(t,0;\tau)=v(t,1;\tau)=0,\,\,\,t\in (\tau,\infty).
    \end{array}
    \right.
\end{equation}
\end{prop}

\begin{proof}
    Firstly, we apply $\partial_{t}$ to $u$ in \eqref{Sol Duhamel}. We get
    \begin{equation}\label{sol Duhamel u_t}
        \partial_t u(t,x)=\partial_t w(t,x) + \int_0^t \partial_t v(t,x;\tau)\d \tau,
    \end{equation}
    where we used the fact that $v(t,x;t)=0$ coming from the initial condition in \eqref{Aux eqn Duhamel}. We differentiate again \eqref{sol Duhamel u_t} with respect to $t$ having in mind that $\partial_t v(t,x;t) = f(t,x)$, we get
    \begin{equation}\label{sol Duhamel u_tt}
        \partial_{tt} u(t,x)=\partial_{tt} w(t,x) + f(t,x) + \int_0^t \partial_{tt} v(t,x;\tau)\d \tau.
    \end{equation}
    Now, the operator $L$ when applied to $u$ in \eqref{Sol Duhamel} gives
    \begin{equation}\label{sol Duhamel L}
        L u(t,x)=L w(t,x) + \int_0^t L v(t,x;\tau)\d \tau.
    \end{equation}
    Multiplying \eqref{sol Duhamel u_t} by $a(x)$ yields
    \begin{equation}\label{sol Duhamel lambdau_t}
        a(x)\partial_t u(t,x)=a(x)\partial_t w(t,x) + \int_0^t a(x)\partial_t v(t,x;\tau)\d \tau.
    \end{equation}
    Combining \eqref{sol Duhamel u_tt}, \eqref{sol Duhamel L} and \eqref{sol Duhamel lambdau_t} and taking into consideration that $w$ and $v$ are the solutions to \eqref{Homog eqn Duhamel} and \eqref{Aux eqn Duhamel}, we arrive at
    \begin{equation*}
        u_{tt}(t,x)+a(x)u_{t}(t,x)+Lu(t,x)=f(t,x).
    \end{equation*}
    Noting that $u(0,x)=w(0,x)=u_0(x)$ from \eqref{Sol Duhamel} and that $u_t(0,x)=\partial_t w(0,x)=u_1(x)$ from \eqref{sol Duhamel u_t} and that from \eqref{Sol Duhamel}, the boundary conditions $u(t,0)=u(t,1)=0$ are satisfied concludes the proof.    
\end{proof}

\section{Classical case: Energy estimates}\label{Classical case}
In this section, we consider the case when the real potential $q$ and the equation coefficients $a$ and $b$ are regular functions. We also assume that $s\geq 0$. In this case, we obtain the well-posedness in the Sobolev spaces $W^s_\mathcal{L}(0,1)$ associated to the operator $\mathcal{L}^s$. We start by proving the well-posedness of our initial/boudary problem \eqref{Equation}-\eqref{boundary conditions} in the case when $a,b \equiv 0$. That is, for the equation
\begin{equation}\label{C.p1}
    \partial^2_t u(t,x)+\mathcal{L}^s u(t,x)=0,\quad (t,x)\in [0,T]\times (0,1),
\end{equation}
with initial conditions
\begin{equation}\label{C.p2}
u(0,x)=u_0(x),\quad 
\partial_t u(0,x)=u_1(x), \quad x\in (0,1),
\end{equation}
and Dirichlet boundary conditions
\begin{equation}\label{C.p3}
u(t,0)=u(t,1)=0,\quad t\in [0,T],
\end{equation}
where  $\mathcal{L}^s$ is defined as in Definition \ref{S.T fractional}.

\begin{thm}\label{Thm: Energy estimates 1}
Assume that $q \in L^{\infty}(0,1)$ is real. For any $s\geq 0$, if the initial data satisfy $(u_0,\, u_1) \in W^{s}_{\mathcal{L}}(0,1)\times L^{2}(0,1)$, then the equation \eqref{C.p1} with the initial/boundary conditions  \eqref{C.p2}-\eqref{C.p3}  has a unique solution $u\in C([0,T], W^{s}_{\mathcal{L}}(0,1))\cap C^1([0,T], L^{2}(0,1))$. It satisfies the estimates
\begin{equation}\label{est1}
    \|u(t,\cdot)\|_{L^2}\lesssim \|u_0\|_{L^2}+\|u_1\|_{W^{-s}_{\mathcal{L}}},
\end{equation}
\begin{equation}\label{est1-1}
    \|u(t,\cdot)\|_{W^{s}_{\mathcal{L}}}\lesssim \|u_0\|_{W^{s}_{\mathcal{L}}}+\|u_1\|_{L^{2}},
\end{equation}
and
\begin{equation}\label{est2}
\|\partial_t u(t,\cdot)\|_{L^2}\lesssim \| u_0\|_{_{W^s_\mathcal{L}}}+\|u_1\|_{L^2},
\end{equation}
where the constants are independent of $u_0$, $u_1$ and $q$.    
\end{thm}
Before giving the proof, one observes that the assumption $q \in L^\infty(0,1)$ fits with \eqref{con-q} since $L^{\infty}(0,1)$ is embedded in $L^{2}(0,1)$.
\begin{proof}
Following the arguments in \cite{RSY22}, we apply the technique of the separation of variables (see, e.g. \cite{GSS17}) to solve the equation \eqref{C.p1} with the initial-boundary conditions \eqref{C.p2}-\eqref{C.p3}. We look for a solution in the form
$$u(t,x)=T(t)X(x),$$
for functions $T(t)$ and $X(x)$ to be determined. Plugging $u(t,x)=T(t)X(x)$ into \eqref{C.p1}, we arrive at the equation
$$T''(t)X(x)+\mathcal{L}^s\left(T(t)X(x)\right)=0,$$
since the operator does not depend on $t$, we obtain
$$T''(t)X(x)+T(t)\mathcal{L}^sX(x)=0.$$
Dividing this equation by $T(t)X(x)$, we get
\begin{equation}\label{3-1}
    \frac{T''(t)}{T(t)}=\frac{-\mathcal{L}^sX(x)}{X(x)}=-\mu,
\end{equation}
for some constant $\mu$. Therefore, if there exists a solution $u(t,x) = T(t)X(x)$ of the wave equation, then $T(t)$ and $X(x)$ must satisfy the equations
$$\frac{T''(t)}{T(t)}=-\mu,$$
$$\frac{\mathcal{L}^sX(x)}{X(x)}=\mu,$$
for some constant $\mu$. In addition, in order for $u$ to satisfy the boundary conditions \eqref{C.p3}, we need our function $X$ to satisfy the boundary conditions \eqref{Dirihle}. That is, we need to find a function $X$ and a scalar $\mu=\lambda^s$, such that
\begin{equation}\label{4}
    \mathcal{L}^sX(x)=\lambda^s X(x),
\end{equation}
\begin{equation}\label{5}
    X(0)=X(1)=0.
\end{equation}

The equation \eqref{4} with the boundary condition \eqref{5} has eigenvalues of the form \eqref{e-val} with corresponding eigenfunctions as in \eqref{sol-SL} of the Sturm-Liouville operator $\mathcal{L}$ generated by the differential expression \eqref{St-L}.

Further, we solve the left hand side of the equation \eqref{3-1} with respect to the independent variable $t$, that is,
\begin{equation}\label{3}
        T''(t)=-\lambda^s T(t),  \qquad t\in [0,T].
\end{equation}

It is well known (\cite{GSS17}) that the solution of the equation \eqref{3} with the initial conditions \eqref{C.p2} is
$$T_n(t)=A_n \cos \left(\sqrt{\lambda_n^s}t\right)+\frac{1}{\sqrt{\lambda_n^s}}B_n \sin\left(\sqrt{\lambda_n^s}t\right),$$
where
$$A_n=\int\limits_0^1u_0(x)\phi_n(x)dx, \qquad B_n=\int\limits_0^1 u_1(x)\phi_n(x)dx.$$

Thus, the solution to \eqref{C.p1}  with the initial/boundary conditions \eqref{C.p2}-\eqref{C.p3} has the form
\begin{equation}\label{23}
    u(t,x)=\sum\limits_{n=1}^\infty \left[A_n \cos \left(\sqrt{\lambda_n^s}t\right)+\frac{1}{\sqrt{\lambda_n^s}}B_n\sin\sqrt{\lambda_n^s} t\right]\phi_n(x). 
\end{equation}

Let us prove that $u\in C^1([0,T],L^2(0,1))$. By using the Cauchy-Schwarz inequality and fixed $t$, we can deduce that
\begin{eqnarray}\label{25}
\|u(t, \cdot)\|^2_{L^2}&=&\int\limits_0^1|u(t,x)|^2dx \nonumber\\
&=&\int\limits_0^1\left|\sum\limits_{n=1}^\infty\left[A_n \cos \sqrt{\lambda_n^s} t+\frac{1}{\sqrt{\lambda_n^s}} B_n\sin\sqrt{\lambda_n^s} t\right]\phi_n(x)\right|^2dx\nonumber\\
&\lesssim& \int\limits_0^1\sum\limits_{n=1}^\infty\left|A_n \cos\sqrt{\lambda_n^s} t+\frac{1}{\sqrt{\lambda_n^s} }B_n\sin\sqrt{\lambda_n^s} t\right|^2|\phi_n(x)|^2dx\nonumber\\
&\leq& \int\limits_0^1\sum\limits_{n=1}^\infty\left(|A_n||\phi_n(x)|+\frac{1}{\sqrt{\lambda_n^s}}|B_n||\phi_n(x)|\right)^2 dx\nonumber\\
&\lesssim& \sum\limits_{n=1}^\infty\left(\int\limits_0^1|A_n|^2|\phi_n(x)|^2dx+\int\limits_0^1\left|\frac{B_n}{\sqrt{\lambda_n^s}}\right|^2|\phi_n(x)|^2dx\right). 
\end{eqnarray}

By using Parseval's identity, we get
$$
\sum\limits_{n=1}^\infty \int\limits_0^1|A_n|^2|\phi_n(x)|^2dx= \sum\limits_{n=1}^\infty |A_n|^2=\int\limits_0^1|u_0(x)|^2dx=\|u_0\|^2_{L^2}.
$$
For the second term in \eqref{25}, using the properties of the eigenvalues of the operator $\mathcal{L}$ and Parseval's identity again, we obtain the following estimate 
\begin{eqnarray}\label{W-1}
\sum\limits_{n=1}^\infty \int\limits_0^1\left|\frac{B_n}{\sqrt{\lambda_n^s}}\right|^2|\phi_n(x)|^2dx&=& \sum\limits_{n=1}^\infty \left|\frac{B_n}{\sqrt{\lambda_n^s}}\right|^2=\sum\limits_{n=1}^\infty \left|\int\limits_0^1\frac{1}{\sqrt{\lambda_n^s}}u_1(x)\phi_n(x)dx\right|^2\nonumber\\
&=&\sum\limits_{n=1}^\infty \left|\int\limits_0^1\mathcal{L}^{-\frac{s}{2}} u_1(x)\phi_n(x)dx\right|^2\nonumber\\
&=&\sum\limits_{n=1}^\infty \left|\mathcal{L}^{-\frac{s}{2}}u_{1,n}\right|^2=\|\mathcal{L}^{-\frac{s}{2}} u_{1}\|^2_{L_2}=\|u_1\|^2_{W^{-s}_\mathcal{L}}.
\end{eqnarray}
Therefore
$$
\|u(t,\cdot)\|^2_{L^2}\lesssim \|u_0\|^2_{L^2}+\|u_1\|^2_{W^{-s}_\mathcal{L}}.
$$
Now, let us estimate $\|\partial_t u(t,\cdot)\|_{L^2}$. We have
\begin{eqnarray}\label{t26}
\|\partial_t u(t,\cdot)\|_{L^2}^2&=&\int\limits_0^1|\partial_tu(t,x)|^2dt\nonumber\\
&=&\int\limits_0^1\left|\sum\limits_{n=1}^\infty\left[-\sqrt{\lambda_n^s}A_n\sin\left(\sqrt{\lambda_n^s}t\right)+\frac{1}{\sqrt{\lambda_n^s}}\sqrt{\lambda_n^s}B_n\cos \sqrt{\lambda_n^s} t\right]\phi_n(x)\right|^2dx \nonumber\\ &\lesssim& \int\limits_0^1\left(\sum\limits_{n=1}^\infty|\sqrt{\lambda_n^s} A_n |^2+\sum\limits_{n=1}^\infty|B_n|^2\right)|\phi_n(x)|^2dx\nonumber\\
&=&\sum\limits_{n=1}^\infty|\sqrt{\lambda_n^s} A_n |^2+\sum\limits_{n=1}^\infty|B_n|^2.
\end{eqnarray}
The second term in \eqref{t26} gives the norm of $\|u_1\|^2_{L^2}$ by Parseval's identity. Now, since $\lambda_n$ are eigenvalues and $\phi_n(x)$ are eigenfunctions of the operator $\mathcal{L}$, we have that 
\begin{eqnarray}\label{21-1}\sum\limits_{n=1}^\infty|\sqrt{\lambda_n^s}A_n|^2&=& \sum\limits_{n=1}^\infty\left|\sqrt{\lambda_n^s}\int\limits_0^1 u_0(x)\phi_n(x)dx\right|^2 =\sum\limits_{n=1}^\infty\left| \int\limits_0^1 \sqrt{\lambda_n^s}u_0(x)\phi_n(x)dx\right|^2\nonumber\\
&=&\sum\limits_{n=1}^\infty\left| \int\limits_0^1 \mathcal{L}^\frac{s}{2}u_0(x)\phi_n(x)dx\right|^2.
      \end{eqnarray}
It follows from Parseval's identity that
\begin{equation}\label{21-2}
    \sum\limits_{n=1}^\infty\left| \int\limits_0^1 \mathcal{L}^\frac{s}{2}u_0(x)\phi_n(x)dx\right|^2=\|\mathcal{L}^\frac{s}{2}u_0\|^2_{L^2}=\|u_0\|^2_{W^s_\mathcal{L}}.
\end{equation}
Thus, 
$$\|\partial_t u(t,\cdot)\|^2_{L^2}\lesssim \|u_0\|^2_{W^s_\mathcal{L}}+\|u_1\|^2_{L^2}.$$
The proof of Theorem \ref{Thm: Energy estimates 1} is then complete.
\end{proof}
In the case when $s=1$, the above estimates can be expressed in terms of all appearing coefficients. This will be needed later on, when the coefficients and data are singular. Let $s=1$. Then the equation \eqref{C.p1} with initial/boundary conditions \eqref{C.p2}-\eqref{C.p3} goes to the explicit form
\begin{equation}\label{without s}
\left\{\begin{array}{l}
\partial^2_t u(t,x)-\partial^2_x u(t,x)+q(x)u(t,x)=0,\quad (t,x)\in [0,T]\times (0,1),   \\
u(0,x)=u_0(x),\quad \partial_t u(0,x)=u_1(x), \quad x\in (0,1), \\
u(t,0)=u(t,1)=0,\quad t\in [0,T].
\end{array}\right.
\end{equation}

\begin{cor}\label{cor1} 
Let $q \in L^\infty(0,1)$ be real, and assume that $u_0 \in L^2(0,1)$ such that $u_0''\in L^2(0,1)$ and that $u_1 \in L^2(0,1)$. Then the problem \eqref{without s} has a unique solution $u\in C([0,T], L^2(0,1))$ which satisfies the estimates
\begin{equation}\label{ec1}
    \|u(t,\cdot)\|_{L^2}\lesssim \|u_0\|_{L^2}+\|u_1\|_{L^2},
\end{equation}
and
\begin{equation}\label{ec1-1}
    \|\partial_t u(t,\cdot)\|_{L^2}\lesssim \|u''_0\|_{L^2}+\|q\|_{L^\infty}\|u_0\|_{L^2}+\|u_1\|_{L^2},
\end{equation}
uniformly in $t\in [0,T]$.
\end{cor}

\begin{proof}
By using the inequality \eqref{25} for $s=1$, we get
\begin{eqnarray}\label{26}
\|u(t, \cdot)\|^2_{L^2}&\lesssim& \sum\limits_{n=1}^\infty\left(\int\limits_0^1|A_n|^2|\phi_n(x)|^2dx+\int\limits_0^1\left|\frac{B_n}{\sqrt{\lambda_n}}\right|^2|\phi_n(x)|^2dx\right).
\end{eqnarray}
According to \eqref{e-val}, we have that $\lambda_n > 1$ for $n=1,\cdots,$ thus
\begin{equation}\label{B-lam}
\int\limits_0^1\left|\frac{B_n}{\sqrt{\lambda_n}}\right|^2|\phi_n(x)|^2dx\leq \int\limits_0^1|B_n|^2|\phi_n(x)|^2dx.
\end{equation}
By using Parseval's identity for \eqref{26} and taking into account the last inequality, we get
$$\|u(t, \cdot)\|^2_{L^2}\lesssim \sum\limits_{n=1}^\infty\left(|A_n|^2+|B_n|^2\right) = \|u_0\|^2_{L^2}+\|u_1\|^2_{L^2},$$
implying \eqref{ec1}. Now, from \eqref{t26} we have that
\begin{equation*}
\|\partial_t u(t,\cdot)\|^2_{L^2} \lesssim \int\limits_0^1\left(\sum\limits_{n=1}^\infty|\sqrt{\lambda_n} A_n |^2+\sum\limits_{n=1}^\infty|B_n|^2\right)dx.
\end{equation*}
The second term of this sum gives the norm of $\|u_1\|^2_{L^2}$ by Parseval's identity. Since $\lambda_n$, for $n=1,\cdots,$ are eigenvalues of the operator $\mathcal{L}$, we obtain 
\begin{eqnarray*}\label{21}\sum\limits_{n=1}^\infty|\sqrt{\lambda_n}A_n|^2&=& \sum\limits_{n=1}^\infty\left|\sqrt{\lambda_n}\int\limits_0^1 u_0(x)\phi_n(x)dx\right|^2 \leq\sum\limits_{n=1}^\infty\left| \int\limits_0^1 \lambda_n u_0(x)\phi_n(x)dx\right|^2\nonumber\\
&=&\sum\limits_{n=1}^\infty\left| \int\limits_0^1 \left(-u''_0(x)+q(x)u_0(x)\right)\phi_n(x)dx\right|^2\\
&\lesssim& \sum\limits_{n=1}^\infty\left| \int\limits_0^1 u''_0(x)\phi_n(x)dx\right|^2+\sum\limits_{n=1}^\infty\left|\int\limits_0^1q(x)u_0(x)\phi_n(x)dx\right|^2.
      \end{eqnarray*}
Again, using Parseval's identity for the second term and using that $q\in L^\infty$, we get
$$\sum\limits_{n=1}^\infty\left|\int\limits_0^1q(x)u_0(x)\phi_n(x)dx\right|^2=\sum\limits_{n=1}^\infty\left|\langle (q u_0),\phi_n\rangle\right|^2=\|q u_0\|^2_{L^2}\leq \|q\|^2_{L^\infty} \|u_0\|^2_{L^2}.$$
Similarly for the first term, we get 
$$\sum\limits_{n=1}^\infty\left| \int\limits_0^1 u''_0(x)\phi_n(x)dx\right|^2=\sum\limits_{n=1}^\infty |u_{0,n}''|^2=\|u_0''\|^2_{L^2},$$
therefore
\begin{eqnarray}\label{LA}\sum\limits_{n=1}^\infty|\sqrt{\lambda_n}A_n|^2&\lesssim& \|u_0''\|^2_{L^2}+\|q\|^2_{L^\infty}\|u_0\|^2_{L^2}.
\end{eqnarray}
Thus, 
$$\|\partial_t u(t,\cdot)\|^2_{L^2}\lesssim \|u''_0\|^2_{L^2}+\|q\|^2_{L^\infty}\|u_0\|^2_{L^2}+\|u_1\|^2_{L^2}.$$
This completes the proof.
\end{proof}

Now we consider the case when $a,b \not\equiv 0$. That is, we consider the problem
\begin{equation}\label{S-L eq in lemma}
     \left\lbrace
    \begin{array}{l}
    \partial_{t}^{2}u(t,x) + \mathcal{L}^su(t,x) + a(x)u(t,x) + b(x)u_{t}(t,x) =0, \quad (t,x)\in [0,T]\times (0,1),\\
    u(0,x)=u_{0}(x), \quad u_{t}(0,x)=u_{1}(x),\\
    u(t,0)=u(t,1)=0, \quad t\in [0,T].
    \end{array}
    \right.
\end{equation}

\begin{thm}\label{Thm energy estimates 2}
    Let $T>0$ and $s\geq 0$. Assume $a,b\in L^{\infty}(0,1)$ to be non-negative, and $q\in L^{\infty}(0,1)$ is real, and let $u_0 \in W_{\mathcal{L}}^s(0,1)$ and $u_1 \in L^2(0,1)$. Then, there exists a unique solution $u\in C([0,T];W_{\mathcal{L}}^{s}(0,1))\cap C^1([0,T];L^{2}(0,1))$ to the problem \eqref{S-L eq in lemma} and it satisfies the estimates
    \begin{align}\label{Energy estimates}
        \bigg\{ \Vert u(t,\cdot)\Vert_{L^2}, & \Vert \mathcal{L}^{\frac{s}{2}}u(t,\cdot)\Vert_{L^2}, \Vert u_t(t,\cdot)\Vert_{L^2} \bigg\} \lesssim\\
        & (1 + \Vert a\Vert_{L^{\infty}} + \Vert b\Vert_{L^{\infty}})\big[\Vert u_0\Vert_{W_{\mathcal{L}}^s} + \Vert u_{1}\Vert_{L^2}\big]. \nonumber
    \end{align}
\end{thm}

\begin{proof}
    By multiplying the equation in \eqref{S-L eq in lemma} by $u_t$ and integrating with respect to the variable $x$ over $[0,1]$, we get
    \begin{equation}\label{Energy funct}
        \langle u_{tt}(t,\cdot),u_{t}(t,\cdot)\rangle_{L^2} + \langle \mathcal{L}^{s}u(t,\cdot),u_{t}(t,\cdot)\rangle_{L^2}
         + \langle a(\cdot)u(t,\cdot),u_{t}(t,\cdot)\rangle_{L^2} + \langle b(\cdot)u_{t}(t,\cdot),u_{t}(t,\cdot)\rangle_{L^2}=0.
    \end{equation}
    It is easy to see that
    \begin{equation}\label{Term1}
        \langle u_{tt}(t,\cdot),u_{t}(t,\cdot)\rangle_{L^2} = \frac{1}{2}\partial_{t}\langle u_{t}(t,\cdot),u_{t}(t,\cdot)\rangle_{L^2} = \frac{1}{2}\partial_{t}\Vert u_{t}(t,\cdot)\Vert_{L^2}^2,
    \end{equation}
    and since the fractional Sturm-Liouville operator is self-adjoint (see Proposition \ref{prop1}) and by the use of the semi-group property \eqref{L^s+s'}, we get
    \begin{align}\label{Term2}
        \langle \mathcal{L}^{s}u(t,\cdot),u_{t}(t,\cdot)\rangle_{L^2} &= \frac{1}{2}\partial_{t}\langle \mathcal{L}^{\frac{s}{2}}u(t,\cdot),\mathcal{L}^{\frac{s}{2}}u_{t}(t,\cdot)\rangle_{L^2} \\
        & = \frac{1}{2}\partial_{t}\Vert \mathcal{L}^{\frac{s}{2}}u(t,\cdot)\Vert_{L^2}^2.\nonumber
    \end{align}
    Moreover, we have
    \begin{equation}\label{Term3}
        \langle a(\cdot)u(t,\cdot),u_{t}(t,\cdot)\rangle_{L^2} = \frac{1}{2}\partial_t\Vert a^{\frac{1}{2}}(\cdot)u(t,\cdot)\Vert_{L^2}^2,
    \end{equation}
    and
    \begin{equation}\label{Term4}
        \langle b(\cdot)u_{t}(t,\cdot),u_{t}(t,\cdot)\rangle_{L^2} = \Vert b^{\frac{1}{2}}(\cdot)u_{t}(t,\cdot)\Vert_{L^2}^2.
    \end{equation}
    By substituting all these terms in \eqref{Energy funct} we arrive at
    \begin{equation}\label{Energy functional1}
        \partial_{t}\Big[ \Vert u_{t}(t,\cdot)\Vert_{L^2}^2 + \Vert \mathcal{L}^{\frac{s}{2}}u(t,\cdot)\Vert_{L^2}^2 + \Vert a^{\frac{1}{2}}(\cdot)u(t,\cdot)\Vert_{L^2}^2\Big] = -2\Vert b^{\frac{1}{2}}(\cdot)u_{t}(t,\cdot)\Vert_{L^2}^2.
    \end{equation}
    By denoting
    \begin{equation}
        E(t):= \Vert u_{t}(t,\cdot)\Vert_{L^2}^2 + \Vert \mathcal{L}^{\frac{s}{2}}u(t,\cdot)\Vert_{L^2}^2 + \Vert a^{\frac{1}{2}}(\cdot)u(t,\cdot)\Vert_{L^2}^2,
    \end{equation}
    it follows that the functional $E(t)$ is decreasing over $[0,1]$ and thus we have $E(t)\leq E(0)$, for all $t\in [0,T]$. We get the estimates
    \begin{align}\label{Estimate u_t}
        \Vert u_{t}(t,\cdot)\Vert_{L^2}^2 & \lesssim \Vert u_{1}\Vert_{L^2}^2 + \Vert \mathcal{L}^{\frac{s}{2}}u_0\Vert_{L^2}^2 + \Vert a^{\frac{1}{2}}u_0\Vert_{L^2}^2\\
        & \lesssim \Vert u_{1}\Vert_{L^2}^2 + \Vert u_0\Vert_{W_{\mathcal{L}}^s}^2 + \Vert a\Vert_{L^{\infty}}\Vert u_0\Vert_{L^2}^2 ,\nonumber
    \end{align}
    and similarly
    \begin{equation}\label{Estimate Lu}
        \Vert \mathcal{L}^{\frac{s}{2}}u(t,\cdot)\Vert_{L^2}^2 \lesssim \Vert u_{1}\Vert_{L^2}^2 + \Vert u_0\Vert_{W_{\mathcal{L}}^s}^2 + \Vert a\Vert_{L^{\infty}}\Vert u_0\Vert_{L^2}^2 ,
    \end{equation}
    and
    \begin{equation}\label{Estimate au}
        \Vert a^{\frac{1}{2}}(\cdot)u(t,\cdot)\Vert_{L^2}^2 \lesssim \Vert u_{1}\Vert_{L^2}^2 + \Vert u_0\Vert_{W_{\mathcal{L}}^s}^2 + \Vert a\Vert_{L^{\infty}}\Vert u_0\Vert_{L^2}^2 .
    \end{equation}
    To estimate the solution $u$, we proceed as follows. We rewrite the equation in \eqref{S-L eq in lemma} as
    \begin{equation}\label{Eq Duham}
        \partial_{t}^{2}u(t,x) + \mathcal{L}^su(t,x) = f(t,x), \quad (t,x)\in [0,T]\times (0,1),
    \end{equation}
    where
    \begin{equation*}
        f(t,x):=-a(x)u(t,x) - b(x)u_{t}(t,x),
    \end{equation*}
    and we apply Duhamel's principle. According to Proposition \ref{Prop Duhamel}, the solution to \eqref{Eq Duham} has the representation
    \begin{equation}\label{Sol Duhamel Thm}
        u(t,x)= w(t,x) + \int_0^t v(t,x;\tau)\d \tau,
    \end{equation}
    where $w(t,x)$ is the solution to the homogeneous problem
    \begin{equation}\label{Homog eqn Duhamel Thm}
        \left\lbrace
        \begin{array}{l}
        w_{tt}(t,x)+\mathcal{L}^s w(t,x)=0, \quad (t,x)\in [0,T]\times (0,1),\\
        w(0,x)=u_{0}(x),\quad w_{t}(0,x)=u_{1}(x),\quad x\in(0,1),\\
        w(t,0)=w(t,1)=0, \quad t\in[0,T].
        \end{array}
        \right.
    \end{equation}
    and $v(t,x;s)$ solves
        \begin{equation}\label{Aux eqn Duhamel Thm}
        \left\lbrace
        \begin{array}{l}
        v_{tt}(t,x;\tau)+\mathcal{L}^s v(t,x;\tau)=0, \quad(t,x)\in (\tau,T)\times (0,1),\\
        v(\tau,x;\tau)=0,\,\,\, v_{t}(\tau,x;\tau)=f(\tau,x),\quad x\in(0,1),\\
        v(t,0;\tau)=v(t,1;\tau)=0, \quad t\in [0,T].
        \end{array}
    \right.
    \end{equation}
    Taking the $L^2$ norm in \eqref{Sol Duhamel Thm} gives
    \begin{equation}
       \Vert u(t,\cdot)\Vert_{L^2} \leq \Vert w(t,\cdot)\Vert_{L^2} + \int_0^t \Vert v(t,\cdot;\tau)\Vert_{L^2}\d \tau,
    \end{equation}
    where we used Minkowski’s integral inequality. The terms on the right hand side can be estimated as follows:
    \begin{equation}\label{Estimate w thm}
        \Vert w(t,\cdot)\Vert_{L^2} \lesssim \|u_0\|_{L^2}+\|u_1\|_{W^{-s}_\mathcal{L}},
    \end{equation}
    which follows from \eqref{est1} since $w$ is the solution to the homogeneous problem and still by \eqref{est1}, the second term is estimated by
    \begin{align}\label{Estimate v thm}
        \Vert v(t,\cdot;\tau)\Vert_{L^2} & \lesssim \Vert f(\tau,\cdot) \Vert_{W^{-s}_\mathcal{L}}\\
        & \leq \Vert a(\cdot)u(\tau,\cdot)\Vert_{W^{-s}_\mathcal{L}} + \Vert b(\cdot)u_{t}(\tau,\cdot) \Vert_{W^{-s}_\mathcal{L}} \nonumber \\
        & \lesssim \Vert a(\cdot)u(\tau,\cdot)\Vert_{L^2} + \Vert b(\cdot)u_{t}(\tau,\cdot) \Vert_{L^2}. \nonumber
    \end{align}
    The last inequality follows from the fact that from \eqref{Estimate u_t} we have that $u_{t}\in L^2(0,1)$, and since $\{\phi_n\}_{n=1,\cdots,}$, the family of eigenfunctions is an orthonormal basis in $L^2(0,1)$, then, $u_t$ can be expanded in terms of this basis as
    \begin{equation}\label{u_t expansion}
        u_t(t,x)=\sum_{n=1}^{\infty}\Tilde{u}_{t,n}(t)\phi_{n}(x), \text{~for~} (t,x)\in [0,T]\times (0,1),
    \end{equation}
    where $\Tilde{u}_{t,n} = \langle u_t ,\phi_n\rangle_{L^2}$ for $n=1,\cdots,$. It follows from \eqref{u_t expansion} that $u_{t}(t,0)=u_{t}(t,1) = 0$ for all $t\in [0,T]$, which allows us to use Proposition \ref{Prop w^-s} as $s\geq 0$. \\
    On the one hand we have
    \begin{align}\label{Estimate v thm1}
        \Vert a(\cdot)u(\tau,\cdot)\Vert_{L^2} & \leq \Vert a\Vert_{L^{\infty}}^{\frac{1}{2}} \Vert a^{\frac{1}{2}}(\cdot)u(\tau,\cdot)\Vert_{L^2}\\
        & \lesssim \Vert a\Vert_{L^{\infty}}^{\frac{1}{2}} \big[\Vert u_{1}\Vert_{L^2}^2 + \Vert u_0\Vert_{W_{\mathcal{L}}^s}^2 + \Vert a\Vert_{L^{\infty}}\Vert u_0\Vert_{L^2}^2\big]^{\frac{1}{2}}, \nonumber
    \end{align}
    which comes from \eqref{Estimate au}. On the other hand
    \begin{align}\label{Estimate v Thm2}
        \Vert b(\cdot)u_{t}(\tau,\cdot)\Vert_{L^2} & \leq \Vert b\Vert_{L^{\infty}} \Vert u_{t}(\tau,\cdot)\Vert_{L^2}\\
        & \lesssim \Vert b\Vert_{L^{\infty}}\big[\Vert u_{1}\Vert_{L^2}^2 + \Vert u_0\Vert_{W_{\mathcal{L}}^s}^2 + \Vert a\Vert_{L^{\infty}}\Vert u_0\Vert_{L^2}^2\big]^{\frac{1}{2}}. \nonumber
    \end{align}
    The latter results from \eqref{Estimate u_t}. We obtain then
    \begin{equation}\label{Estimate v thm3}
        \Vert v(t,\cdot;\tau)\Vert_{L^2} \lesssim (1 + \Vert a\Vert_{L^{\infty}} + \Vert b\Vert_{L^{\infty}})\big[\Vert u_0\Vert_{L^2} + \Vert u_0\Vert_{W_{\mathcal{L}}^s} + \Vert u_{1}\Vert_{L^2} \big].
    \end{equation}
    We substitute \eqref{Estimate w thm} and \eqref{Estimate v thm3} into \eqref{Sol Duhamel Thm}, we get our estimate for $u$,
    \begin{align}
        \Vert u(t,\cdot)\Vert_{L^2} &\lesssim (1 + \Vert a\Vert_{L^{\infty}} + \Vert b\Vert_{L^{\infty}})\big[\Vert u_0\Vert_{L^2} + \Vert u_0\Vert_{W_{\mathcal{L}}^s} + \Vert u_{1}\Vert_{L^2} + \Vert u_{1}\Vert_{W_{\mathcal{L}}^{-s}} \big]\nonumber\\
        &\lesssim (1 + \Vert a\Vert_{L^{\infty}} + \Vert b\Vert_{L^{\infty}})\big[\Vert u_0\Vert_{W_{\mathcal{L}}^s} + \Vert u_{1}\Vert_{L^2}\big].
    \end{align}
    The last estimate follows by Proposition \ref{Prop w^-s}. This completes the proof of the theorem.  
\end{proof}
In the case when $s=1$, the estimates in Theorem \ref{Thm energy estimates 2} can be expressed in terms of all coefficients appearing in \eqref{S-L eq in lemma}, including the potential $q$. This removes the dependence of the estimates on $\mathcal{L}$.

\begin{cor}\label{cor2}
    Let $T>0$. Assume $a,b\in L^{\infty}(0,1)$ to be non-negative, and $q\in L^{\infty}(0,1)$ is real. Let $u_0 \in L^2(0,1)$ be such that $u_0'' \in L^2(0,1)$, and let $u_1 \in L^2(0,1)$. Then, the problem
    \begin{equation}\label{S-L eq in corollary}
        \left\lbrace
        \begin{array}{l}
        \partial_{t}^{2}u(t,x) + \mathcal{L}u(t,x) + a(x)u(t,x) + b(x)u_{t}(t,x) =0, \quad (t,x)\in [0,T]\times (0,1),\\
        u(0,x)=u_{0}(x), \quad u_{t}(0,x)=u_{1}(x),\\
        u(t,0)=u(t,1)=0, \quad t\in [0,T],
        \end{array}
        \right.
    \end{equation}
    has a unique solution $u\in C([0,T];W_{\mathcal{L}}^{1}(0,1))\cap C^1([0,T];L^{2}(0,1))$ and it satisfies
    \begin{align}\label{Energy estimates cor}
        \bigg\{ \Vert u(t,\cdot)\Vert_{L^2}, &\Vert \mathcal{L}^{\frac{1}{2}}u(t,\cdot)\Vert_{L^2}, \Vert u_t(t,\cdot)\Vert_{L^2} \bigg\} \lesssim\\ 
        &(1 + \Vert q\Vert_{L^{\infty}})(1 + \Vert a\Vert_{L^{\infty}})(1 + \Vert b\Vert_{L^{\infty}})\big[\Vert u_0\Vert_{L^2} + \Vert u''_0\Vert_{L^2} + \Vert u_{1}\Vert_{L^2} \big]. \nonumber
    \end{align}
\end{cor}

\begin{proof}
We consider $s=1$ and argue similarly as in the proof of Theorem \ref{Thm energy estimates 2}. Firstly, we get the estimates
    \begin{align}\label{Estimates corollary}
        \bigg\{ \Vert u_{t}(t,\cdot)\Vert_{L^2} , \Vert \mathcal{L}^{\frac{1}{2}}u(t,\cdot)\Vert_{L^2} ,& \Vert a^{\frac{1}{2}}(\cdot)u(t,\cdot)\Vert_{L^2} \bigg\}\\
        & \lesssim \Vert u_{1}\Vert_{L^2} + \Vert u_0\Vert_{W_{\mathcal{L}}^1} + \Vert a\Vert_{L^{\infty}}^{\frac{1}{2}}\Vert u_0\Vert_{L^2}\nonumber\\
        & \lesssim \Vert u_{1}\Vert_{L^2} + \Vert u^{''}_0\Vert_{L^2} + \Vert q\Vert_{L^{\infty}}\Vert u_0\Vert_{L^2} + \Vert a\Vert_{L^{\infty}}^{\frac{1}{2}}\Vert u_0\Vert_{L^2},\nonumber
    \end{align}
where we used that for $s=1$ and by arguing as in \eqref{21-1} and \eqref{21-2}, the term $\Vert u_0\Vert_{W_{\mathcal{L}}^1}$ can be estimated by
\begin{eqnarray*}
\Vert u_0\Vert^2_{W_{\mathcal{L}}^1} &=& \Vert\mathcal{L}^{\frac{1}{2}}u_0\Vert^2_{L^2} =\sum\limits_{n=1}^n\left|\sqrt{\lambda_n}A_n\right|^2\leq\sum\limits_{n=1}^n\left|\lambda_nA_n\right|^2 = \Vert\mathcal{L}u_0\Vert^2_{L^2}\\
&\leq& \Big(\Vert u^{''}_0\Vert_{L^2} + \Vert q\Vert_{L^{\infty}}\Vert u_0\Vert_{L^2}\Big)^2,
\end{eqnarray*}
    which is valid since $\lambda_n >1$ for $n=1,2,\cdots,$ resulting from \eqref{e-val}. Once again, to estimate $u$, we rewrite the equation in \eqref{S-L eq in corollary} as
    \begin{equation}\label{Eq Duham 1}
        \partial_{t}^{2}u(t,x) + \mathcal{L}u(t,x) = f(t,x), \quad (t,x)\in [0,T]\times (0,1),
    \end{equation}
    where
    \begin{equation*}
        f(t,x):=-a(x)u(t,x) - b(x)u_{t}(t,x),
    \end{equation*}
    and we apply Duhamel's principle to get the following representation for the solution
    \begin{equation}\label{Sol Duhamel cor}
        u(t,x)= w(t,x) + \int_0^t v(t,x;\tau)\d \tau,
    \end{equation}
    where $w(t,x)$ is the solution to the homogeneous problem
    \begin{equation}\label{Homog eqn Duhamel cor}
        \left\lbrace
        \begin{array}{l}
        w_{tt}(t,x)+\mathcal{L} w(t,x)=0, \quad (t,x)\in [0,T]\times (0,1),\\
        w(0,x)=u_{0}(x),\quad w_{t}(0,x)=u_{1}(x),\quad x\in(0,1),\\
        w(t,0)=w(t,1)=0, \quad t\in[0,T].
        \end{array}
        \right.
    \end{equation}
    and $v(t,x;s)$ solves
        \begin{equation}\label{Aux eqn Duhamel cor}
        \left\lbrace
        \begin{array}{l}
        v_{tt}(t,x;\tau)+\mathcal{L}v(t,x;\tau)=0, \quad(t,x)\in (\tau,T)\times (0,1),\\
        v(\tau,x;\tau)=0,\,\,\, v_{t}(\tau,x;\tau)=f(\tau,x),\quad x\in(0,1),\\
        v(t,0;\tau)=v(t,1;\tau)=0, \quad t\in [0,T].
        \end{array}
    \right.
    \end{equation}
    By using the estimate \eqref{ec1} from Corollary \ref{cor1} to estimate $w$ and $v$ in combination with \eqref{Estimates corollary} and proceeding as in Theorem \ref{Thm energy estimates 2} we easily get
    \begin{equation*}
        \Vert u(t,\cdot)\Vert_{L^2}\lesssim (1 + \Vert q\Vert_{L^{\infty}})(1 + \Vert a\Vert_{L^{\infty}})(1 + \Vert b\Vert_{L^{\infty}})\big[\Vert u_0\Vert_{L^2} + \Vert u''_0\Vert_{L^2} + \Vert u_{1}\Vert_{L^2} \big],
    \end{equation*}
    \begin{equation*}
        \Vert \mathcal{L}^{\frac{1}{2}}u(t,\cdot)\Vert_{L^2} \lesssim (1 + \Vert q\Vert_{L^{\infty}})(1 + \Vert a\Vert_{L^{\infty}})(1 + \Vert b\Vert_{L^{\infty}})\big[\Vert u_0\Vert_{L^2} + \Vert u''_0\Vert_{L^2} + \Vert u_{1}\Vert_{L^2} \big],
    \end{equation*}
    and
    \begin{equation*}
        \Vert u_t(t,\cdot)\Vert_{L^2} \lesssim (1 + \Vert q\Vert_{L^{\infty}})(1 + \Vert a\Vert_{L^{\infty}})(1 + \Vert b\Vert_{L^{\infty}})\big[\Vert u_0\Vert_{L^2} + \Vert u''_0\Vert_{L^2} + \Vert u_{1}\Vert_{L^2} \big],
    \end{equation*}
    ending the proof.
\end{proof}

\section{Very weak well-posedness}\label{VW well-posed}
For $s\geq 0$ and $T>0$, we consider the initial/boundary problem
\begin{equation}\label{S-L eq VWS}
     \left\lbrace
    \begin{array}{l}
    \partial_{t}^{2}u(t,x) + \mathcal{L}^su(t,x) + a(x)u(t,x) + b(x)u_{t}(t,x) =0, \quad (t,x)\in [0,T]\times (0,1),\\
    u(0,x)=u_{0}(x), \quad u_{t}(0,x)=u_{1}(x),\\
    u(t,0)=u(t,1)=0, \quad t\in [0,T].
    \end{array}
    \right.
\end{equation}
Now, we want to analyse solutions to \eqref{S-L eq VWS} for less regular coefficients $a,b,q$ and initial data $u_0,u_1$ having in mind distributions. To obtain the well-posedness in such cases, we will be using the concept of very weak solutions. To start with, for $\varepsilon\in(0,1]$ we consider families of regularised problems to \eqref{S-L eq VWS} arising from the regularising nets
\begin{equation}
    (a_\varepsilon)_{\varepsilon}=(a\ast \psi_\varepsilon)_{\varepsilon}, ~(b_\varepsilon)_{\varepsilon}=(b\ast \psi_\varepsilon)_{\varepsilon}, ~(q_\varepsilon)_{\varepsilon}=(q\ast \psi_\varepsilon)_{\varepsilon},\label{regularising nets1}
\end{equation}
and
\begin{equation}
    (u_{0,\varepsilon})_{\varepsilon}=(u_{0}\ast \psi_\varepsilon)_{\varepsilon}, ~ (u_{1,\varepsilon})_{\varepsilon}=(u_{1}\ast \psi_\varepsilon)_{\varepsilon},\label{regularising nets2}
\end{equation}
where
$\psi_{\varepsilon}(x)=\varepsilon^{-1}\psi(x/\varepsilon)$. The function $\psi$ is a Friedrichs-mollifier, i.e. $\psi\in C_{0}^{\infty}(\mathbb{R}^{d})$, $\psi\geq 0$ and $\int\psi =1$. We introduce the following definitions.
\begin{defn}[\textbf{Moderateness}]\label{Defn moderateness}
Let $X$ be a normed space of functions on $\mathbb R$ endowed with the norm $\Vert \cdot\Vert_X$.
\begin{enumerate}
    \item[\textsc{1.}] A net of functions $(f_\varepsilon)_{\varepsilon\in(0,1]}$ from $X$ is said to be $X$-moderate, if there exist $N\in\mathbb{N}_0$ such that
\begin{equation}\label{Defn moderateness1}
    \Vert f_\varepsilon\Vert_X \lesssim \varepsilon^{-N},
\end{equation}
and in particular
    \item[\textsc{2.}] a net of functions $(f_\varepsilon)_{\varepsilon\in(0,1]}$ from $L^2(0,1)$ is said to be $H^2$-moderate, if there exist $N\in\mathbb{N}_0$ such that
\begin{equation}\label{Defn moderateness1.1}
    \Vert f_\varepsilon\Vert_{L^2} + \Vert f''_\varepsilon\Vert_{L^2} \lesssim \varepsilon^{-N}.
\end{equation}
    \item[\textsc{3.}] For $T>0$, $s\geq 0$ and $q$ smooth enough, a net of functions $(u_\varepsilon(\cdot,\cdot))_{\varepsilon\in(0,1]}$ from $C([0,T];W_{\mathcal{L}}^{s}(0,1))\cap C^1([0,T];L^{2}(0,1))$ is said to be $C([0,T];W_{\mathcal{L}}^{s}(0,1))\cap C^1([0,T];L^{2}(0,1))$-moderate and we shortly write uniformly $s$-moderate, if there exist $N\in\mathbb{N}_0$ such that
    \begin{equation}\label{Defn moderateness2}
        \sup_{t\in [0,T]}\Vert u_\varepsilon(t,\cdot)\Vert_s \lesssim \varepsilon^{-N}.
    \end{equation}
    \item[\textsc{4.}] For $T>0$ and $s=1$. A net of functions $(u_\varepsilon(\cdot,\cdot))_{\varepsilon\in(0,1]}$ from $C([0,T];W_{\mathcal{L}}^{1}(0,1))\cap C^1([0,T];L^{2}(0,1))$ is said to be $C([0,T];W_{\mathcal{L}}^{1}(0,1))\cap C^1([0,T];L^{2}(0,1))$-moderate and we shortly write uniformly $1$-moderate, if there exist $N\in\mathbb{N}_0$ such that
    \begin{equation}\label{Defn moderateness2.1}
        \sup_{t\in [0,T]}\Vert u_\varepsilon(t,\cdot)\Vert_1 \lesssim \varepsilon^{-N}.
    \end{equation}
\end{enumerate}
\end{defn}
\begin{rem} We note that such assumptions are natural for distributional coefficients in the sense that regularisations of distributions are moderate. Precisely, by the structure theorems for distributions (see, e.g. \cite{FJ98}, \cite{Gar21}), we know that 
\begin{equation}\label{moder}
  \mathcal{D}'(0,1) \subset \{L^p(0,1) -\text{moderate families} \},  
\end{equation}
for any $p \in [1,\infty)$, which means that a solution to an initial/boundary problem may not exist in the sense of distributions, while it may exist in the set of $L^p$-moderate functions. 
\end{rem}
For instance, if we consider $f\in L^2(0,1)$, $f:(0,1)\to \mathbb{C}$. We define 
$$\Tilde{f}=\left\{\begin{array}{l}
    f, \text{ on }(0,1),  \\
    0,  \text{ on }\mathbb{R} \setminus (0,1).
\end{array}\right.$$
We have then $\Tilde{f}:\mathbb{R}\to \mathbb{C}$, and $\Tilde{f}\in \mathcal{E}'(\mathbb{R}).$

Let $\Tilde{f}_\varepsilon=\Tilde{f}*\psi_\varepsilon$ be obtained via convolution of $\Tilde{f}$ with a mollifying net $\psi_\varepsilon$, where 
$$\psi_\varepsilon(x)=\frac{1}{\varepsilon}\psi\left(\frac{x}{\varepsilon}\right),\quad \text{for}\,\, \psi\in C^\infty_0(\mathbb{R}),\, \int \psi=1. $$
Then the regularising net $(\Tilde{f}_\varepsilon)$ is $L^p$-moderate for any $p \in [1,\infty)$, and it approximates $f$ on $(0,1)$:
$$0\leftarrow \|\Tilde{f}_\varepsilon-\Tilde{f}\|^p_{L^p(\mathbb{R})}\approx \|\Tilde{f}_\varepsilon-f\|^p_{L^p(0,1)}+\|\Tilde{f}_\varepsilon\|^p_{L^p(\mathbb{R}\setminus (0,1))}.$$
In order to prove uniqueness of very weak solutions to \eqref{S-L eq VWS}, we will need the following definition.
\begin{defn}[\textbf{Negligibility}]\label{Defn negligibility}
    Let $X$ be a normed space with the norm $\Vert\cdot\Vert_X$. A net of functions $(f_\varepsilon)_{\varepsilon\in(0,1]}$ from $X$ is said to be $X$-negligible, if the estimate
    \begin{equation}\label{Negligibility formula}
        \Vert f_\varepsilon\Vert_X \leq C_k\varepsilon^k,
    \end{equation}
    is valid for all $k>0$, where $C_k$ may depend on $k$. We shortly write $\Vert f_\varepsilon\Vert_X \lesssim \varepsilon^k$. In particular, a net a functions $(f_\varepsilon)_{\varepsilon\in(0,1]}$ from $L^2$ such that $(f''_\varepsilon)_{\varepsilon\in(0,1]}\in L^2(0,1)$ is said to be $H^2$-negligible, if the estimate
    \begin{equation}\label{Negligibility formula1}
        \Vert f_\varepsilon\Vert_{L^2} + \Vert f''_\varepsilon\Vert_{L^2}
    \leq C_k\varepsilon^k,
    \end{equation}
    is valid for all $k>0$.
\end{defn}
Now, we are ready to introduce the notion of very weak solutions adapted to our problem. We treat two cases. Firstly, we treat the case when $s\geq 0$ and the potential $q$ is smooth enough. We also treat the case when $s=1$ where $q$ is allowed to be singular.

\subsection{Case 1: $s\geq 0$} 
We should mention that the reason we consider here regular potential $q$, lies in the fact that the estimates obtained in Theorem \ref{Thm energy estimates 2} depend on $\mathcal{L}$.
\begin{defn}[Very weak solution]\label{D2}
Let $a,b,u_0,u_1\in \mathcal{D}'(0,1)$ be such that $a,b$ are non-negative (in the sense of their representatives) and assume $q\in L^\infty(0,1)$ is real. A net of functions $(u_\varepsilon)_{\varepsilon\in (0,1]}$ is said to be a very weak solution to the initial/boundary problem  \eqref{S-L eq VWS} if there exist non-negative $L^\infty$-moderate regularisations $(a_{\varepsilon})_{\varepsilon}$ and $(b_\varepsilon)_{\varepsilon}$ of $a$ and $b$, a $ W_{\mathcal{L}}^{s}(0,1)$-moderate regularisation $(u_{0,\varepsilon})_{\varepsilon}$ of $u_0$ and an $L^2(0,1)$-moderate regularisation $(u_{1,\varepsilon})_{\varepsilon}$ of $u_1$ such that the family $(u_\varepsilon)_{\varepsilon}$ solves the $\varepsilon$-dependent problems
\begin{equation}\label{vw1}
\left\{\begin{array}{l}\partial^2_t u_\varepsilon(t,x)+\mathcal{L}^su_\varepsilon(t,x)+a_{\varepsilon}(x)u_\varepsilon(t,x)+b_\varepsilon(x)\partial_tu_\varepsilon(t,x)=0,\quad (t,x)\in [0,T]\times(0,1),\\
u_\varepsilon(0,x)=u_{0,\varepsilon}(x),\quad \partial_tu_\varepsilon(0,x)=u_{1,\varepsilon}(x),\quad x\in (0,1), \\
u_\varepsilon(t,0)=0=u_\varepsilon(t,1), \quad t\in[0,T],
\end{array}\right.\end{equation}
for any $\varepsilon\in (0,1]$, and $(u_\varepsilon)_{\varepsilon}$ is uniformly $s$-moderate.
\end{defn}

\begin{thm}[Existence]\label{Thm existence 1}
    Let $a,b,u_0,u_1$ and $q$ as in Definition \ref{D2}. Then the initial/boundary problem \eqref{S-L eq VWS} has a very weak solution.
\end{thm}

\begin{proof}
    Since $a,b,u_0,u_1$ are moderate, then, there exists $N_1, N_2, N_3, N_4 \in\mathbb N$, such that
    \begin{equation*}
        \Vert a_{\varepsilon}\Vert_{L^{\infty}} \lesssim \varepsilon^{-N_1},\quad \Vert b_{\varepsilon}\Vert_{L^{\infty}} \lesssim \varepsilon^{-N_2},
    \end{equation*}
    and 
    \begin{equation*}
        \Vert u_{0,\varepsilon}\Vert_{W_{\mathcal{L}}^{s}} \lesssim \varepsilon^{-N_3},\quad \Vert u_{1,\varepsilon}\Vert_{L^2} \lesssim \varepsilon^{-N_4}.
    \end{equation*}
    Using the estimate \eqref{Energy estimates}, we get
    \begin{equation*}
        \Vert u_{\varepsilon}(t,\cdot)\Vert_{1} \lesssim \varepsilon^{-\max \{N_1,N_2\}-\max \{N_3,N_4\}},
    \end{equation*}
    for all $t\in[0,T]$. Thus, $(u_{\varepsilon})_{\varepsilon}$ is uniformly $s$-moderate and the existence of a very weak solution follows, ending the proof.
\end{proof}
The uniqueness of the very weak solution is proved in the sense of the following definition.
\begin{defn}[Uniqueness of very weak solutions]\label{D4}
We say that the initial/boundary problem \eqref{S-L eq VWS} has a unique very weak solution, if
for any non-negative $L^\infty$-moderate nets $(a_\varepsilon)_{\varepsilon}$, $(\Tilde{a}_\varepsilon)_{\varepsilon}$, $(b_\varepsilon)_{\varepsilon}$, $(\Tilde{b}_\varepsilon)_{\varepsilon}$, such that $(a_\varepsilon-\Tilde{a}_\varepsilon)_{\varepsilon}$ and $(b_\varepsilon-\Tilde{b}_\varepsilon)_{\varepsilon}$ are $L^\infty$-negligible; for any $W_{\mathcal{L}}^{s}$-moderate regularisations $(u_{0,\varepsilon},\,\Tilde{u}_{0,\varepsilon})_{\varepsilon}$, such that $(u_{0,\varepsilon}-\Tilde{u}_{0,\varepsilon})_{\varepsilon}$ is $W_{\mathcal{L}}^{s}$-negligible and for any $L^2$-moderate regularisations $(u_{1,\varepsilon},\,\Tilde{u}_{1,\varepsilon})_{\varepsilon}$, such that $(u_{1,\varepsilon}-\Tilde{u}_{1,\varepsilon})_{\varepsilon}$ is $L^2$-negligible, we have that $(u_\varepsilon-\Tilde{u}_\varepsilon)_{\varepsilon}$ is $L^2$-negligible for all $t\in[0,T]$, where $(u_{\varepsilon})_{\varepsilon}$ and $(\Tilde{u}_{\varepsilon})_{\varepsilon}$ are the families of solutions to the corresponding regularised problems
\begin{equation}\label{un1}
\left\{\begin{array}{l}\partial^2_t u_\varepsilon(t,x)+\mathcal{L}^su_\varepsilon(t,x)+a_{\varepsilon}(x)u_\varepsilon(t,x)+b_\varepsilon(x)\partial_tu_\varepsilon(t,x)=0,\quad (t,x)\in [0,T]\times(0,1),\\
u_\varepsilon(0,x)=u_{0,\varepsilon}(x),\quad \partial_tu_\varepsilon(0,x)=u_{1,\varepsilon}(x),\quad x\in (0,1), \\
u_\varepsilon(t,0)=0=u_\varepsilon(t,1), \quad t\in[0,T],\end{array}\right.
\end{equation}
and
\begin{equation}\label{un2}
    \left\{\begin{array}{l}\partial^2_t \Tilde{u}_\varepsilon(t,x)+\mathcal{L}^s \Tilde{u}_\varepsilon(t,x)+\Tilde{a}_\varepsilon(x)\Tilde{u}_\varepsilon(t,x)+\Tilde{b}_\varepsilon(x) \partial_t\Tilde{u}_\varepsilon(t,x)=0,\quad (t,x)\in [0,T]\times(0,1),\\
 \Tilde{u}_\varepsilon(0,x)=\Tilde{u}_{0,\varepsilon}(x),\quad \partial_t \Tilde{u}_\varepsilon(0,x)=\Tilde{u}_{1,\varepsilon}(x), \quad x\in (0,1), \\
\Tilde{u}_\varepsilon(t,0)=0=\Tilde{u}_\varepsilon(t,1), \quad t\in[0,T],
\end{array}\right.
\end{equation}
respectively.
\end{defn}

\begin{thm}[Uniqueness]\label{Thm1 uniqueness}
    Let $a,b,u_0,u_1\in \mathcal{D}'(0,1)$ be such that $a,b$ are non-negative, and assume $q\in L^\infty(0,1)$ is real. In the background of Theorem \ref{Thm existence 1}, the very weak solution to the initial/boundary problem \eqref{S-L eq VWS} is unique.
\end{thm}

\begin{proof}
    Let $(u_{\varepsilon})_{\varepsilon}$ and $(\Tilde{u}_{\varepsilon})_{\varepsilon}$ be the nets of solutions to \eqref{un1} and \eqref{un2} corresponding to the families of regularised coefficients and initial data $\big( a_\varepsilon,b_\varepsilon,u_{0,\varepsilon},u_{1,\varepsilon} \big)_\varepsilon$ and $\big( \Tilde{a}_\varepsilon,\Tilde{b}_\varepsilon,\Tilde{u}_{0,\varepsilon},\Tilde{u}_{1,\varepsilon} \big)_\varepsilon$ respectively. Assume that the nets $(a_\varepsilon-\Tilde{a}_\varepsilon)_{\varepsilon}$ and $(b_\varepsilon-\Tilde{b}_\varepsilon)_{\varepsilon}$ are $L^\infty$-negligible; $(u_{0,\varepsilon}-\Tilde{u}_{0,\varepsilon})_{\varepsilon}$ is $W_{\mathcal{L}}^{s}$-negligible and $(u_{1,\varepsilon}-\Tilde{u}_{1,\varepsilon})_{\varepsilon}$ is $L^2$-negligible. Let us introduce
    \begin{equation*}
        U_{\varepsilon}(t,x):=u_{\varepsilon}(t,x)-\Tilde{u}_{\varepsilon}(t,x),
    \end{equation*}
    then, $U_{\varepsilon}(t,x)$ is solution to
    \begin{equation}\label{Equation1 U}
        \left\{\begin{array}{l}
        \partial_{t}^2U_{\varepsilon}(t,x) + \mathcal{L}^{s}U_{\varepsilon}(t,x) + a_{\varepsilon}(x)U_{\varepsilon}(t,x) + b_{\varepsilon}(x)\partial_{t}U_{\varepsilon}(t,x)=f_{\varepsilon}(t,x),\\
        U_{\varepsilon}(0,x)=(u_{0,\varepsilon}-\Tilde{u}_{0,\varepsilon})(x),\quad \partial_{t}U_{\varepsilon}(0,x)=(u_{1,\varepsilon}-\Tilde{u}_{1,\varepsilon})(x),\\
        U_{\varepsilon}(t,0)=U_{\varepsilon}(t,1)=0, 
        \end{array}\right.
    \end{equation}
    for $(t,x)\in [0,T]\times(0,1)$, where
    \begin{equation*}
        f_{\varepsilon}(t,x):=\big(\Tilde{a}_{\varepsilon}(x)-a_{\varepsilon}(x)\big)\Tilde{u}_{\varepsilon}(t,x) + \big(\Tilde{b}_{\varepsilon}(x)-b_{\varepsilon}(x)\big)\partial_{t}\Tilde{u}_{\varepsilon}(t,x).
    \end{equation*}
    By using Duhamel's principle, $U_{\varepsilon}(t,x)$ is given by
    \begin{equation}\label{Duhamel1 solution}
        U_{\varepsilon}(t,x) = W_{\varepsilon}(t,x) + \int_{0}^{t}V_{\varepsilon}(t,x;\tau)\d \tau,
    \end{equation}
    where $W_{\varepsilon}(t,x)$ is the solution to the problem
    \begin{equation}\label{Equation1 W}
        \left\{
        \begin{array}{l}
        \partial_{t}^2W_{\varepsilon}(t,x) + \mathcal{L}^{s}W_{\varepsilon}(t,x) + a_{\varepsilon}(x)W_{\varepsilon}(t,x) + b_{\varepsilon}(x)\partial_{t}W_{\varepsilon}(t,x)=0,\\
        W_{\varepsilon}(0,x)=(u_{0,\varepsilon}-\Tilde{u}_{0,\varepsilon})(x),\quad \partial_{t}U_{\varepsilon}(0,x)=(u_{1,\varepsilon}-\Tilde{u}_{1,\varepsilon})(x),\\
        W_{\varepsilon}(t,0)=W_{\varepsilon}(t,1)=0, 
        \end{array}\right.
    \end{equation}
    for $(t,x)\in [0,T]\times(0,1)$, and $V_{\varepsilon}(t,x;s)$ solves
    \begin{equation}\label{Equation1 V}
        \left\{
        \begin{array}{l}
        \partial_{t}^2V_{\varepsilon}(t,x;\tau) + \mathcal{L}^{s}V_{\varepsilon}(t,x;\tau) + a_{\varepsilon}(x)V_{\varepsilon}(t,x;\tau) + b_{\varepsilon}(x)\partial_{t}V_{\varepsilon}(t,x;\tau)=0,\\
        V_{\varepsilon}(\tau,x;\tau)=0,\quad \partial_{t}V_{\varepsilon}(\tau,x;\tau)=f_{\varepsilon}(\tau,x),\\
        V_{\varepsilon}(t,0;\tau)=V_{\varepsilon}(t,1;\tau)=0, 
        \end{array}\right.
    \end{equation}
    for $(t,x)\in [\tau,T]\times(0,1)$ and $s\in [0,T]$.
    By taking the $L^2$-norm in both sides in \eqref{Duhamel1 solution} we get
    \begin{equation}\label{Duhamel1 solution estimate}
        \Vert U_{\varepsilon}(t,\cdot)\Vert_{L^2} \leq \Vert W_{\varepsilon}(t,\cdot)\Vert_{L^2} + \int_{0}^{t}\Vert V_{\varepsilon}(t,\cdot;\tau)\Vert_{L^2}\d \tau.
    \end{equation}
    Using \eqref{Energy estimates} to estimate $\Vert W_{\varepsilon}(t,\cdot)\Vert_{L^2}$ and $\Vert V_{\varepsilon}(t,\cdot;\tau)\Vert_{L^2}$, we get
    \begin{equation*}
        \Vert W_{\varepsilon}(t,\cdot)\Vert_{L^2} \lesssim \Big(1+\Vert a_{\varepsilon}\Vert_{L^{\infty}}+\Vert b_{\varepsilon}\Vert_{L^{\infty}}\Big)\bigg[\Vert u_{0,\varepsilon}-\Tilde{u}_{0,\varepsilon}\Vert_{W_{\mathcal{L}}^{s}} + \Vert u_{1,\varepsilon}-\Tilde{u}_{1,\varepsilon}\Vert_{L^2}\bigg],
    \end{equation*}
    and
    \begin{equation*}
        \Vert V_{\varepsilon}(t,\cdot;\tau)\Vert_{L^2} \lesssim \Big(1+\Vert a_{\varepsilon}\Vert_{L^{\infty}}+\Vert b_{\varepsilon}\Vert_{L^{\infty}}\Big)\bigg[\Vert f_{\varepsilon}(\tau,\cdot)\Vert_{L^2}\bigg].
    \end{equation*}
    It follows from \eqref{Duhamel1 solution estimate} that
    \begin{align}\label{Estimate U}
        \Vert U_{\varepsilon}(t,\cdot)\Vert_{L^2} \lesssim \Big(1+\Vert a_{\varepsilon}\Vert_{L^{\infty}}+\Vert b_{\varepsilon}\Vert_{L^{\infty}}\Big)&\bigg[\Vert u_{0,\varepsilon}-\Tilde{u}_{0,\varepsilon}\Vert_{W_{\mathcal{L}}^{s}}\\
        & + \Vert u_{1,\varepsilon}-\Tilde{u}_{1,\varepsilon}\Vert_{L^2} + \int_0^T \Vert f_{\varepsilon}(\tau,\cdot)\Vert_{L^2}\d \tau \bigg],
    \end{align}
    since $t\in [0,T]$. Let us estimate $\Vert f_{\varepsilon}(\tau,\cdot)\Vert_{L^2}$. We have,
    \begin{align}\label{Estimate f_epsilon}
        \Vert f_{\varepsilon}(\tau,\cdot)\Vert_{L^2} & \leq \Vert (\Tilde{a}_{\varepsilon}(\cdot)-a_{\varepsilon}(\cdot))\Tilde{u}_{\varepsilon}(\tau,\cdot)\Vert_{L^2} + \Vert(\Tilde{b}_{\varepsilon}(\cdot)-b_{\varepsilon}(\cdot))\partial_{t}\Tilde{u}_{\varepsilon}(\tau,\cdot)\Vert_{L^2}\\
        & \nonumber \lesssim \Vert \Tilde{a}_{\varepsilon}-a_{\varepsilon}\Vert_{L^{\infty}}\Vert\Tilde{u}_{\varepsilon}(\tau,\cdot)\Vert_{L^2 } + \Vert\Tilde{b}_{\varepsilon}-b_{\varepsilon}\Vert_{L^{\infty}}\Vert\partial_{t}\Tilde{u}_{\varepsilon}(\tau,\cdot)\Vert_{L^2 }.
    \end{align}
    Thus, we get
    \begin{align}\label{Estimate U 1}
        \Vert U_{\varepsilon}(t,\cdot)\Vert_{L^2} \lesssim & \Big(1+\Vert a_{\varepsilon}\Vert_{L^{\infty}}+\Vert b_{\varepsilon}\Vert_{L^{\infty}}\Big)\bigg[\Vert u_{0,\varepsilon}-\Tilde{u}_{0,\varepsilon}\Vert_{W_{\mathcal{L}}^{s}} + \Vert u_{1,\varepsilon}-\Tilde{u}_{1,\varepsilon}\Vert_{L^2}\\
        & \nonumber+ \Vert \Tilde{a}_{\varepsilon}-a_{\varepsilon}\Vert_{L^{\infty}}\int_0^T \Vert\Tilde{u}_{\varepsilon}(\tau,\cdot)\Vert_{L^2 }\d \tau + \Vert \Tilde{b}_{\varepsilon}-b_{\varepsilon}\Vert_{L^{\infty}}\int_0^T \Vert\partial_t\Tilde{u}_{\varepsilon}(\tau,\cdot)\Vert_{L^2 }\d \tau\bigg].
    \end{align}
    Now, using the fact that $(a_{\varepsilon})_{\varepsilon}$ and $(b_{\varepsilon})_{\varepsilon}$ are $L^{\infty}$-moderate by assumption, and that the net $(\Tilde{u}_{\varepsilon})_{\varepsilon}$ is uniformly $s$-moderate being a very weak solution to \eqref{S-L eq VWS} this on one hand and from the other hand that the nets $(a_\varepsilon-\Tilde{a}_\varepsilon)_{\varepsilon}$ and $(b_\varepsilon-\Tilde{b}_\varepsilon)_{\varepsilon}$ are $L^\infty$-negligible; $(u_{0,\varepsilon}-\Tilde{u}_{0,\varepsilon})_{\varepsilon}$ is $W_{\mathcal{L}}^{s}$-negligible and $(u_{1,\varepsilon}-\Tilde{u}_{1,\varepsilon})_{\varepsilon}$ is $L^2$-negligible, it follows from \eqref{Estimate U 1} that
    \begin{equation*}
        \Vert U_{\varepsilon}(t,\cdot)\Vert_{L^2} \lesssim \varepsilon^{k},
    \end{equation*}
    for all $k>0$. This completes the proof.    
\end{proof}

\subsection{Case 2: $s=1$} In this case, the energy estimates obtained in Theorem \ref{Thm energy estimates 2} are expressed in terms of all appearing coefficients and initial data including the potential $q$ as it was shown in Corollary \ref{cor2}. this allows us to consider singular potentials. So, the problem to be analysed here is the initial/boundary problem 
\begin{equation}\label{eq_s=0}
\left\lbrace
\begin{array}{l}
\partial_{t}^{2}u(t,x) +\mathcal{L}_qu(t,x) ++ a(x)u(t,x) + b(x)u_{t}(t,x) =0,\quad (t,x)\in [0,T]\times (0,1),\\
u(0,x)=u_{0}(x), \quad u_{t}(0,x)=u_{1}(x),\quad x\in(0,1),\\
u(t,0)=u(t,1)=0, \quad t\in [0,T],
\end{array}
\right.   
\end{equation}
where
\begin{equation}
    \mathcal{L}_qu(t,x):=-\partial_{x}^{2}u(t,x) + q(x)u(t,x).
\end{equation}
Here, the coefficients $a$, $b$ and the initial data $u_0,u_1$ together with the potential $q$ are assumed to be distributions on $(0,1)$. Let us first adapt our previous definitions to this case.
\begin{defn}[Very weak solution]\label{Defn VWS case-2}
Let $a,b,q,u_0,u_1\in \mathcal{D}'(0,1)$ be such that $a,b$ are non-negative and assume that $q\in L^\infty(0,1)$ is real. A net of functions $(u_\varepsilon)_{\varepsilon\in (0,1]}$ is said to be a very weak solution to the initial/boundary problem  \eqref{eq_s=0} if there exist non-negative $L^\infty$-moderate regularisations $(a_\varepsilon)_{\varepsilon}$, $(b_\varepsilon)_{\varepsilon}$ of $a,b$, and $(q_\varepsilon)_{\varepsilon}$ of $q$, an $H^2$-moderate regularisation $(u_{0,\varepsilon})_{\varepsilon}$ of $u_0$ and an $L^2$-moderate regularisation $(u_{1,\varepsilon})_{\varepsilon}$ of $u_1$ such that the family $(u_\varepsilon)_{\varepsilon}$ solves the $\varepsilon$-dependent problems
\begin{equation}\label{vw1-case2}
\left\{\begin{array}{l}\partial^2_t u_\varepsilon(t,x)+ \mathcal{L}_{q_\varepsilon}u_\varepsilon(t,x) + a_{\varepsilon}(x)u_\varepsilon(t,x)+b_\varepsilon(x)\partial_tu_\varepsilon(t,x)=0,\quad (t,x)\in [0,T]\times(0,1),\\
u_\varepsilon(0,x)=u_{0,\varepsilon}(x),\quad \partial_tu_\varepsilon(0,x)=u_{1,\varepsilon}(x),\quad x\in (0,1), \\
u_\varepsilon(t,0)=0=u_\varepsilon(t,1), \quad t\in[0,T],
\end{array}\right.\end{equation}
for any $\varepsilon\in (0,1]$ and $(u_\varepsilon)_{\varepsilon}$ is $1$-moderate.
\end{defn}
\begin{thm}[Existence]\label{Ext_s=0}
Let $a,b,q,u_0,u_1$ be as in Definition \ref{Defn VWS case-2}. Assume that there exist non-negative $L^\infty$-moderate regularisations $(a_\varepsilon)_{\varepsilon}$, $(b_\varepsilon)_{\varepsilon}$ of $a,b$, and $(q_\varepsilon)_{\varepsilon}$ of $q$, an $H^2$-moderate regularisation $(u_{0,\varepsilon})_{\varepsilon}$ of $u_0$ and an $L^2$-moderate regularisation $(u_{1,\varepsilon})_{\varepsilon}$ of $u_1$. Then, the initial/boundary problem  \eqref{eq_s=0} has a very weak solution.
\end{thm}
\begin{proof}
    Since $a,b,q,u_0,u_1$ are moderate, this means that there exists $N_1, N_2, N_3, N_4, N_5 \in\mathbb N$, such that
    \begin{equation*}
        \Vert a_{\varepsilon}\Vert_{L^{\infty}} \lesssim \varepsilon^{-N_1},\quad \Vert b_{\varepsilon}\Vert_{L^{\infty}} \lesssim \varepsilon^{-N_2},\quad \Vert q_{\varepsilon}\Vert_{L^{\infty}} \lesssim \varepsilon^{-N_3},
    \end{equation*}
    and 
    \begin{equation*}
        \Vert u_{0,\varepsilon}\Vert_{L^2} + \Vert u''_{0,\varepsilon}\Vert_{L^2} \lesssim \varepsilon^{-N_4},\quad \Vert u_{1,\varepsilon}\Vert_{L^2} \lesssim \varepsilon^{-N_5}.
    \end{equation*}
    From \eqref{Energy estimates cor}, we have
    \begin{eqnarray*}
        \Vert u_{\varepsilon}(t,\cdot)\Vert_{1}&=&\Vert u_{\varepsilon}(t,\cdot)\Vert_{L^2}+\Vert \mathcal{L}^{\frac{1}{2}}u_{\varepsilon}(t,\cdot)\Vert_{L^2}+ \Vert u_{t,\varepsilon}(t,\cdot)\Vert_{L^2} \\ 
        &\lesssim&(1 + \Vert q_{\varepsilon}\Vert_{L^{\infty}})(1 + \Vert a_{\varepsilon}\Vert_{L^{\infty}})(1 + \Vert b_{\varepsilon}\Vert_{L^{\infty}})\big[\Vert u_{0,\varepsilon}\Vert_{L^2} + \Vert u''_{0,\varepsilon}\Vert_{L^2} + \Vert u_{1,\varepsilon}\Vert_{L^2} \big]\\
        &\lesssim& (1+\varepsilon^{-N_3})(1+\varepsilon^{-N_1})(1+\varepsilon^{-N_2})\big[\varepsilon^{-N_4}+\varepsilon^{-N_5}\big]\\
        &\lesssim&\varepsilon^{-\max\{N_1,N_2,N_3\}-\max\{N_4,N_5\}},
    \end{eqnarray*}
for all $t\in[0,T]$. Thus, $(u_{\varepsilon})_{\varepsilon}$ is $C^1$-moderate and the existence of a very weak solution follows.    
\end{proof}

In order to prove the uniqueness of the very weak solution in the case when $s=1$, we need to adapt Definition \ref{D4} to this case. The definition reads,

\begin{defn}[Uniqueness of very weak solutions]\label{D4-1}
We say that the initial/boundary problem \eqref{eq_s=0} has a unique very weak solution, if
for any non-negative $L^\infty$-moderate nets $(a_\varepsilon)_{\varepsilon}$, $(\Tilde{a}_\varepsilon)_{\varepsilon}$, $(b_\varepsilon)_{\varepsilon}$, $(\Tilde{b}_\varepsilon)_{\varepsilon}$, and real $(q_\varepsilon)_{\varepsilon}$, $(\Tilde{q}_\varepsilon)_{\varepsilon}$, such that $(a_\varepsilon-\Tilde{a}_\varepsilon)_{\varepsilon}$, $(b_\varepsilon-\Tilde{b}_\varepsilon)_{\varepsilon}$ and $(q_\varepsilon-\Tilde{q}_\varepsilon)_{\varepsilon}$ are $L^\infty$-negligible; for any $H^2$-moderate regularisations $(u_{0,\varepsilon},\,\Tilde{u}_{0,\varepsilon})_{\varepsilon}$ such that $(u_{0,\varepsilon}-\Tilde{u}_{0,\varepsilon})_{\varepsilon}$ is $H^2$-negligible and for any $L^2$-moderate regularisations $(u_{1,\varepsilon},\,\Tilde{u}_{1,\varepsilon})_{\varepsilon}$, such that $(u_{1,\varepsilon}-\Tilde{u}_{1,\varepsilon})_{\varepsilon}$ is $L^2$-negligible, we have that $(u_\varepsilon-\Tilde{u}_\varepsilon)_{\varepsilon}$ is $L^2$-negligible for all $t\in[0,T]$, where $(u_{\varepsilon})_{\varepsilon}$ and $(\Tilde{u}_{\varepsilon})_{\varepsilon}$ are the families of solutions to the corresponding regularised problems
\begin{equation}\label{un1-1}
\left\{\begin{array}{l}\partial^2_t u_\varepsilon(t,x)+\mathcal{L}_{q_\varepsilon}u_\varepsilon(t,x)+a_{\varepsilon}(x)u_\varepsilon(t,x)+b_\varepsilon(x)\partial_tu_\varepsilon(t,x)=0,\quad (t,x)\in [0,T]\times(0,1),\\
u_\varepsilon(0,x)=u_{0,\varepsilon}(x),\quad \partial_tu_\varepsilon(0,x)=u_{1,\varepsilon}(x),\quad x\in (0,1), \\
u_\varepsilon(t,0)=0=u_\varepsilon(t,1), \quad t\in[0,T],\end{array}\right.
\end{equation}
and
\begin{equation}\label{un2-1}
    \left\{\begin{array}{l}\partial^2_t \Tilde{u}_\varepsilon(t,x)+\mathcal{L}_{\Tilde{q}_\varepsilon} \Tilde{u}_\varepsilon(t,x)+\Tilde{a}_\varepsilon(x)\Tilde{u}_\varepsilon(t,x)+\Tilde{b}_\varepsilon(x) \partial_t\Tilde{u}_\varepsilon(t,x)=0,\quad (t,x)\in [0,T]\times(0,1),\\
 \Tilde{u}_\varepsilon(0,x)=\Tilde{u}_{0,\varepsilon}(x),\quad \partial_t \Tilde{u}_\varepsilon(0,x)=\Tilde{u}_{1,\varepsilon}(x), \quad x\in (0,1), \\
\Tilde{u}_\varepsilon(t,0)=0=\Tilde{u}_\varepsilon(t,1), \quad t\in[0,T],
\end{array}\right.
\end{equation}
respectively.
\end{defn}


\begin{thm}[Uniqueness]\label{Thm2 uniqueness}
    Let $a,b,q,u_0,u_1\in \mathcal{D}'(0,1)$. Under the assumptions of of Theorem \ref{Ext_s=0}, the very weak solution to the initial/boundary problem \eqref{eq_s=0} is unique.
\end{thm}

\begin{proof}
    Let $(u_{\varepsilon})_{\varepsilon}$ and $(\Tilde{u}_{\varepsilon})_{\varepsilon}$ be the nets of solutions to \eqref{un1-1} and \eqref{un2-1} corresponding to the families of regularised coefficients and initial data $\big( a_\varepsilon,b_\varepsilon,q_\varepsilon,u_{0,\varepsilon},u_{1,\varepsilon} \big)_\varepsilon$ and $\big( \Tilde{a}_\varepsilon,\Tilde{b}_\varepsilon,\Tilde{q}_\varepsilon,\Tilde{u}_{0,\varepsilon},\Tilde{u}_{1,\varepsilon} \big)_\varepsilon$ respectively. Assume that the nets $(a_\varepsilon-\Tilde{a}_\varepsilon)_{\varepsilon}$, $(b_\varepsilon-\Tilde{b}_\varepsilon)_{\varepsilon}$ and $(q_\varepsilon-\Tilde{q}_\varepsilon)_{\varepsilon}$ are $L^\infty$-negligible; $(u_{0,\varepsilon}-\Tilde{u}_{0,\varepsilon})_{\varepsilon}$ is $H^2$-negligible and $(u_{1,\varepsilon}-\Tilde{u}_{1,\varepsilon})_{\varepsilon}$ is $L^2$-negligible. Then, $(U_{\varepsilon}(t,x))_{\varepsilon}:= (u_{\varepsilon}(t,x)-\Tilde{u}_{\varepsilon}(t,x))_\varepsilon$ is solution to
    \begin{equation}\label{Equation1 U-1}
        \bigg\{
        \begin{array}{l}
        \partial_{t}^2U_{\varepsilon}(t,x) + \mathcal{L}_{q_\varepsilon}U_{\varepsilon}(t,x) + a_{\varepsilon}(x)U_{\varepsilon}(t,x) + b_{\varepsilon}(x)\partial_{t}U_{\varepsilon}(t,x)=f_{\varepsilon}(t,x),\\
        U_{\varepsilon}(0,x)=(u_{0,\varepsilon}-\Tilde{u}_{0,\varepsilon})(x),\quad \partial_{t}U_{\varepsilon}(0,x)=(u_{1,\varepsilon}-\Tilde{u}_{1,\varepsilon})(x),\\
        U_{\varepsilon}(t,0)=U_{\varepsilon}(t,1)=0, 
        \end{array}
    \end{equation}
    for $(t,x)\in [0,T]\times(0,1))$, where,
    \begin{equation*}
        f_{\varepsilon}(t,x):=\Big[\big(\Tilde{a}_{\varepsilon}(x)-a_{\varepsilon}(x)\big) + \big(\Tilde{q}_{\varepsilon}(x)-q_{\varepsilon}(x)\big)\Big]\Tilde{u}_{\varepsilon}(t,x) + \big(\Tilde{b}_{\varepsilon}(x)-b_{\varepsilon}(x)\big)\partial_{t}\Tilde{u}_{\varepsilon}(t,x).
    \end{equation*}
    Thanks to Duhamel's principle, $U_{\varepsilon}(t,x)$ can be represented as
    \begin{equation}\label{Duhamel1 solution-1}
        U_{\varepsilon}(t,x) = W_{\varepsilon}(t,x) + \int_{0}^{t}V_{\varepsilon}(t,x;\tau)\d \tau,
    \end{equation}
    where $W_{\varepsilon}(t,x)$ is the solution to the problem
    \begin{equation}\label{Equation1 W-1}
        \bigg\{
        \begin{array}{l}
        \partial_{t}^2W_{\varepsilon}(t,x) + \mathcal{L}_{q_\varepsilon}W_{\varepsilon}(t,x) + a_{\varepsilon}(x)W_{\varepsilon}(t,x) + b_{\varepsilon}(x)\partial_{t}W_{\varepsilon}(t,x)=0,\\
        W_{\varepsilon}(0,x)=(u_{0,\varepsilon}-\Tilde{u}_{0,\varepsilon})(x),\quad \partial_{t}U_{\varepsilon}(0,x)=(u_{1,\varepsilon}-\Tilde{u}_{1,\varepsilon})(x),\\
        W_{\varepsilon}(t,0)=W_{\varepsilon}(t,1)=0, 
        \end{array}
    \end{equation}
    for $(t,x)\in [0,T]\times(0,1)$, and $V_{\varepsilon}(t,x;\tau)$ solves
    \begin{equation}\label{Equation1 V-1}
        \bigg\{
        \begin{array}{l}
        \partial_{t}^2V_{\varepsilon}(t,x;\tau) + \mathcal{L}_{q_\varepsilon}V_{\varepsilon}(t,x;\tau) + a_{\varepsilon}(x)V_{\varepsilon}(t,x;\tau) + b_{\varepsilon}(x)\partial_{t}V_{\varepsilon}(t,x;\tau)=0,\\
        V_{\varepsilon}(\tau,x;\tau)=0,\quad \partial_{t}V_{\varepsilon}(\tau,x;\tau)=f_{\varepsilon}(\tau,x),\\
        V_{\varepsilon}(t,0;\tau)=V_{\varepsilon}(t,1;\tau)=0, 
        \end{array}
    \end{equation}
    for $(t,x)\in [\tau,T]\times(0,1)$ and $\tau\in [0,T]$. Using the estimate \eqref{Energy estimates cor} and reasoning similarly as in the proof of Theorem \ref{Thm1 uniqueness}, we arrive at
    \begin{align}\label{Estimate U-1}
        \Vert U_{\varepsilon}(t,\cdot)\Vert_{L^2} \lesssim  \Big(1+\Vert a_{\varepsilon}\Vert_{L^{\infty}}\Big)&\Big(1 + \Vert b_{\varepsilon}\Vert_{L^{\infty}}\Big)\Big(1 + \Vert b_{\varepsilon}\Vert_{L^{\infty}}\Big)\bigg[\Vert u_{0,\varepsilon}-\Tilde{u}_{0,\varepsilon}\Vert_{L^2}\\
        & + \Vert u''_{0,\varepsilon}-\Tilde{u}''_{0,\varepsilon}\Vert_{L^2} + \Vert u_{1,\varepsilon}-\Tilde{u}_{1,\varepsilon}\Vert_{L^2} + \int_0^T \Vert f_{\varepsilon}(\tau,\cdot)\Vert_{L^2}\d \tau \bigg],
    \end{align}
    and we easily show that $\Vert f_{\varepsilon}(\tau,\cdot)\Vert_{L^2}$ can be estimated by
    \begin{align}\label{Estimate f_epsilon-1}
        \Vert f_{\varepsilon}(\tau,\cdot)\Vert_{L^2} \leq \Big(\Vert \Tilde{a}_{\varepsilon}-a_{\varepsilon}\Vert_{L^{\infty}} + \Vert \Tilde{q}_{\varepsilon}-q_{\varepsilon}\Vert_{L^{\infty}}\Big)\Vert\Tilde{u}_{\varepsilon}(\tau,\cdot)\Vert_{L^2 } + \Vert\Tilde{b}_{\varepsilon}-b_{\varepsilon}\Vert_{L^{\infty}}\Vert\partial_{t}\Tilde{u}_{\varepsilon}(\tau,\cdot)\Vert_{L^2}.
    \end{align}
    Combining \eqref{Estimate U-1} and \eqref{Estimate f_epsilon-1} and using the moderateness and negligibility assumptions, one can easily see that
    \begin{equation*}
        \Vert U_{\varepsilon}(t,\cdot)\Vert_{L^2} \lesssim \varepsilon^{k},
    \end{equation*}
    for all $k>0$ and $t\in [0,T]$, showing the uniqueness of the very weak solution.    
\end{proof}

\section{Consistency with classical theory}\label{Consistency}
We conclude this article with the important question of proving that the classical solutions to the initial/boundary problem \eqref{Equation}, as given in Theorem \ref{Thm energy estimates 2} and Corollary \ref{cor2}, can be recaptured by the very weak solutions as $\varepsilon\rightarrow 0$. We prove the following theorems in both cases: the case when $s\geq 0$ and the case when $s=1$.

\begin{thm}[Consistency, case: $s\geq 0$]\label{Thm Consistency 1}
    Let $T>0$ and $s\geq 0$. Assume $a,b\in L^{\infty}(0,1)$ to be non-negative and $q\in L^{\infty}(0,1)$ is real and let $u_0 \in W_{\mathcal{L}}^s(0,1)$ and $u_1 \in L^2(0,1)$, in such way that a classical solution to \eqref{Equation} exists. Then, for any regularising families $(a_{\varepsilon})_{\varepsilon}$, $(b_{\varepsilon})_{\varepsilon}$ for the equation coefficients, satisfying
    \begin{equation}\label{approx.condition1}
        \lVert a_{\varepsilon} - a\rVert_{L^{\infty}} \rightarrow 0\quad\text{and}\quad \lVert b_{\varepsilon} - b\rVert_{L^{\infty}} \rightarrow 0, \text{~as~} \varepsilon \rightarrow 0,
    \end{equation} 
    and any regularising families $(u_{0,\varepsilon})_{\varepsilon}$ and $(u_{1,\varepsilon})_{\varepsilon}$ for the initial data, satisfying
    \begin{equation}\label{approx.condition2}
        \Vert u_{0,\varepsilon}-u_{0}\Vert_{W_{\mathcal{L}}^{s}} \rightarrow 0 \quad \text{and}\quad \Vert u_{1,\varepsilon}-u_{1}\Vert_{L^2} \rightarrow 0\text{~as~} \varepsilon \rightarrow 0,
    \end{equation}
    the net $(u_{\varepsilon})_{\varepsilon}$ converges to the classical solution of the initial/boundary problem (\ref{Equation}) in $L^{2}$ as $\varepsilon \rightarrow 0$.
\end{thm}

\begin{proof}
    Let $u$ be the classical solution to \eqref{Equation} and $(u_\varepsilon)_\varepsilon$ its very weak solution. Then, for $\varepsilon\in (0,1]$, $U_{\varepsilon}(t,x):= u_{\varepsilon}(t,x)-u(t,x)$ is solution to
    \begin{equation}\label{cons1}
        \left\{\begin{array}{l}\partial^2_t U_\varepsilon(t,x)+\mathcal{L}^sU_\varepsilon(t,x)+a_{\varepsilon}(x)U_\varepsilon(t,x)+b_\varepsilon(x)\partial_tU_\varepsilon(t,x)=f_{\varepsilon}(t,x),\\
        U_{\varepsilon}(0,x)=(u_{0,\varepsilon}-u_0)(x),\quad \partial_{t}U_{\varepsilon}(0,x)=(u_{1,\varepsilon}-u_1)(x),\\
        U_\varepsilon(t,0)=U_\varepsilon(t,1)=0,
        \end{array}\right.
    \end{equation}
    where $(t,x)\in [0,T]\times(0,1)$ and
    \begin{equation*}
        f_{\varepsilon}(t,x):=-\big(a_{\varepsilon}(x)-a(x)\big)u(t,x) - \big(b_{\varepsilon}(x)-b(x)\big)\partial_{t}u(t,x).
    \end{equation*}
    By arguing as we did in Theorem \ref{Thm1 uniqueness}, we obtain
    \begin{align}\label{Estimate U 3}
        \Vert U_{\varepsilon}(t,\cdot)\Vert_{L^2} \lesssim & \Big(1+\Vert a_{\varepsilon}\Vert_{L^{\infty}}+\Vert b_{\varepsilon}\Vert_{L^{\infty}}\Big)\bigg[\Vert u_{0,\varepsilon}-u_{0}\Vert_{W_{\mathcal{L}}^{s}} + \Vert u_{1,\varepsilon}-u_{1}\Vert_{L^2}\\
        & \nonumber+ \Vert a_{\varepsilon}-a\Vert_{L^{\infty}}\int_0^T \Vert u(\tau,\cdot)\Vert_{L^2 }\d \tau + \Vert b_{\varepsilon}-b\Vert_{L^{\infty}}\int_0^T \Vert\partial_t u(\tau,\cdot)\Vert_{L^2 }\d \tau\bigg],
    \end{align}
    uniformly in $t\in [0,T]$. Since
    \begin{equation*}
        \lVert a_{\varepsilon} - a\rVert_{L^{\infty}} \rightarrow 0,~ \lVert b_{\varepsilon} - b\rVert_{L^{\infty}} \rightarrow 0, ~ \Vert u_{0,\varepsilon}-u_{0}\Vert_{W_{\mathcal{L}}^{s}} \rightarrow 0~\text{and}~ \Vert u_{1,\varepsilon}-u_{1}\Vert_{L^2} \rightarrow 0
    \end{equation*}
    as $\varepsilon \rightarrow 0$ by assumption and the terms $\Vert a_{\varepsilon}\Vert_{L^{\infty}}$, $\Vert b_{\varepsilon}\Vert_{L^{\infty}}$, $\Vert u(\tau,\cdot)\Vert_{L^2 }$ and $\Vert\partial_t u(\tau,\cdot)\Vert_{L^2 }$ are bounded, it follows that
    \begin{equation*}
        \Vert U_{\varepsilon}(t,\cdot)\Vert_{L^2} \rightarrow 0,\quad \text{as } \varepsilon\rightarrow 0,
    \end{equation*}
    uniformly in $t\in [0,T]$. This completes the proof of Theorem \ref{Thm Consistency 1}.    
\end{proof}
In the case when $s=1$, the consistency theorem reads as following.

\begin{thm}[Consistency, case: $s=1$]\label{Thm Consistency 2}
    Let $T>0$. Assume $a,b\in L^{\infty}(0,1)$ to be non-negative and that $q\in L^{\infty}(0,1)$ is real. Let $u_0 \in L^2(0,1)$ such that $u_0'' \in L^2(0,1)$ and $u_1 \in L^2(0,1)$, in such way that a classical solution to \eqref{Equation} exists. Then, for any regularising families $(a_{\varepsilon})_{\varepsilon}$, $(b_{\varepsilon})_{\varepsilon}$ and $(q_{\varepsilon})_{\varepsilon}$ for the equation coefficients, satisfying
    \begin{equation}\label{approx.condition3}
        \lVert a_{\varepsilon} - a\rVert_{L^{\infty}} \rightarrow 0,\quad \lVert b_{\varepsilon} - b\rVert_{L^{\infty}} \rightarrow 0 \quad\text{and~} \lVert q_{\varepsilon} - q\rVert_{L^{\infty}} \rightarrow 0, \text{~as~} \varepsilon \rightarrow 0,
    \end{equation} 
    and any regularising families $(u_{0,\varepsilon})_{\varepsilon}$ and $(u_{1,\varepsilon})_{\varepsilon}$ for the initial data, the net $(u_{\varepsilon})_{\varepsilon}$ converges to the classical solution of the initial/boundary problem (\ref{Equation}) in $L^{2}$ as $\varepsilon \rightarrow 0$.
\end{thm}

\begin{proof}
    For $u$ being the classical solution to \eqref{Equation} and $(u_\varepsilon)_\varepsilon$ the very weak solution, by reasoning as in Theorem \ref{Thm Consistency 1}, we obtain
    \begin{align*}
        \Vert U_{\varepsilon}(t,\cdot)\Vert_{L^2} \lesssim & \Big(1+\Vert a_{\varepsilon}\Vert_{L^{\infty}}\Big)\Big(1 + \Vert b_{\varepsilon}\Vert_{L^{\infty}}\Big)\Big(1 + \Vert b_{\varepsilon}\Vert_{L^{\infty}}\Big)\bigg[\Vert u_{0,\varepsilon}-u_{0}\Vert_{L^2}  + \Vert u''_{0,\varepsilon}-u''_{0}\Vert_{L^2}\\
        & +\Vert u_{1,\varepsilon}-u_{1}\Vert_{L^2} + \Big(\Vert a_{\varepsilon}-a\Vert_{L^{\infty}} + \Vert q_{\varepsilon}-q\Vert_{L^{\infty}}\Big)\int_0^T \Vert u(\tau,\cdot)\Vert_{L^2 }\d \tau\\
        & + \Vert b_{\varepsilon}-b\Vert_{L^{\infty}}\int_0^T \Vert\partial_{t}u(\tau,\cdot)\Vert_{L^2}\d \tau\bigg].
    \end{align*}
    It follows from assumptions and that
    \begin{equation*}
        \Vert u_{0,\varepsilon}-u_{0}\Vert_{L^2} \rightarrow 0,\quad \Vert u''_{0,\varepsilon}-u''_{0}\Vert_{L^2} \rightarrow 0\quad \text{and} \quad\Vert u_{1,\varepsilon}-u_{1}\Vert_{L^2} \rightarrow 0 \text{~as~} \varepsilon \rightarrow 0,
    \end{equation*}
    and $u$ being the classical solution to \eqref{Equation} that
    \begin{equation*}
        \Vert U_{\varepsilon}(t,\cdot)\Vert_{L^2} \rightarrow 0,\quad \text{as } \varepsilon\rightarrow 0,
    \end{equation*}
    uniformly in $t\in [0,T]$, ending the proof.    
\end{proof}

\end{document}